\documentclass[11pt,a4wide]{amsart}
\usepackage{a4wide}
\usepackage{pst-all}
\usepackage{pstcol}
\usepackage{graphics,graphicx}
\usepackage[english]{babel}

\newtheorem{thm}{Theorem}[section]
\newtheorem{cor}[thm]{Corollary}
\newtheorem{lem}[thm]{Lemma}
\newtheorem{prop}[thm]{Proposition}
\newtheorem{defn}[thm]{Definition}

\newtheorem{rem}[thm]{Remark}
\numberwithin{equation}{section}

\begin{document}

\title[Control and stabilization of the linearized Benjamin equation]{On the controllability and stabilization of the linearized Benjamin equation on a periodic domain}

\author{ Mahendra  Panthee  $\&$ Francisco J. Vielma Leal}

\address{Department of Mathematics,  University of Campinas, S\~ao Paulo, SP, Brasil}

\email{ mpanthee@ime.unicamp.br, \;fvielmaleal7@gmail.com}

\thanks{This work was partially supported by  FAPESP, Brazil with grants 2016/25864-6 and 2015/06131-5.}

\thanks{}




\dedicatory{}

\commby{Mahendra Panthee}


\begin{abstract}
 In this work we study the controllability and stabilization of the linearized Benjamin equation which models the unidirectional propagation of long waves in a two-fluid system where the lower fluid with greater density is infinitely deep and the interface is subject to capillarity. We show that the linearized Benjamin equation with periodic boundary conditions is exactly controllable and  exponentially stabilizable with any given decay rate in $H_{p}^{s}(\mathbb{T})$ with $s\geq0$.
\end{abstract}

\maketitle

\section{Introduction}

We consider the Benjamin equation,
\begin{equation}\label{BO}
    \partial_{t}u-\alpha \mathcal{H}\partial^{2}_{x}u-\partial^{3}_{x}u+\partial_{x}u^{2}=0, \;\;\;\;x\in \mathbb{R},\;\;t\in \mathbb{R},
\end{equation}
where $u=u(x,t)$ denotes a real-valued function of two real variables $x$ and $t,$
$\alpha$ is a positive real number, and $\mathcal{H}$ denotes the Hilbert transform defined by
\begin{equation}\label{hilberttramsf}
 \mathcal{H}(f)(x)=\frac{1}{\pi}p.v.\int \frac{f(x-y)}{y}dy.
\end{equation}

The Benjamin equation \eqref{BO} is an integro-differential equation that serves as a generic model for
unidirectional propagations of long waves in a two-fluid system where the lower fluid with
greater density is infinitely deep and the interface is subject to capillarity. It was derived by Benjamin
\cite{5} to study gravity-capillarity surface waves of solitary type in deep water.
He also showed that solutions of the equation \eqref{BO} satisfy the conserved quantities,
$$I_{1}(u)=\frac{1}{2}\int_{\mathbb{R}} u^{2}(x,t)\;dx,$$
and
$$\displaystyle{I_{2}(u)=\int_{\mathbb{R}} \left[\frac{1}{2}(\partial_{x}u)^{2}(x,t)-
\frac{\alpha}{2}u(x,t)\mathcal{H}\partial_{x}u(x,t)-\frac{1}{3}u^{3}(x,t)\right]\;dx}.$$

 Several works have been devoted to study the existence, stability and asymptotic properties of  solitary
waves solutions of \eqref{BO}, see for instance ~\cite{15,16,5,18}. The well-posedness
of the initial value  problem (IVP) associated to the Benjamin equation on $H^{s}(\mathbb{R})$  has been intensively studied for many years, see ~\cite{19,20,22,23}.
The best known global well-posedness result in $L^{2}(\mathbb{R})$ is due to Linares \cite{23}. There are further  improvements of this result, viz., local well-posedness in $H^{s}(\mathbb{R})$ for $s\geq -\frac{3}{4}$ \cite{20}.

The Benjamin equation posed on a periodic spatial domain $\mathbb{T}:=\mathbb{R}/(2\pi \mathbb{Z})$ is also widely studied in the literature.  Linares \cite{23} proved global well-posedness in $L^{2}(\mathbb{T})$, and   Shi and Junfeng \cite{22}  proved local well-posedness in $H^{s}(\mathbb{T})$ for $s\geq -\frac{1}{2}$. 

In this work, we interested in considering the linearized Benjamin equation posed on a periodic domain,
\begin{equation}\label{2D-BO}
  \partial_{t}u-\alpha \mathcal{H}\partial^{2}_{x}u-\partial^{3}_{x}u
=0, \;\;\;\;x\in\mathbb{T},\;\;t\in\mathbb{R},
\end{equation}
and study controllability and stabilization. More precisely, we  are interested in the following two problems.

\noindent
\textbf{\emph{Exact control problem:}}
Given an initial state $u_{0}$ and a terminal state $u_{1}$
in a certain space with $[u_{0}]=[u_{1}],$ can one find an appropriate control input
$f$ so that the equation 
\begin{equation}\label{2D-BO1}
  \partial_{t}u-\alpha \mathcal{H}\partial^{2}_{x}u-\partial^{3}_{x}u
=f(x,t), \;\;\;\;x\in \mathbb{T},\;\;t\in \mathbb{R},
\end{equation}
admits a solution $u$ such that $u(x,0)=u_{0}(x)$ and $u(x,T)=u_{1}(x) \;\;\text{for all}\;\; x\in \mathbb{T}$ and any final time $T>0$?

\noindent
\textbf{\emph{Stabilization Problem:}}
Given $u_{0}$ in a certain space.
Can one find a feedback control law: $f=Ku$ so that the resulting closed-loop system
\begin{equation}\label{2D-BO2}
  \partial_{t}u-\alpha \mathcal{H}\partial^{2}_{x}u-\partial^{3}_{x}u
=Ku, \;\;u(x,0)=u_{0},\;\;x\in \mathbb{T},\;\;t\in \mathbb{R}^{+}
\end{equation}
is asymptotically stable  as $t\rightarrow \infty$?

Control and stabilization of the dispersive equations has been widely studied in the literature. In particular, for the Korteweg-de Vries (KdV) equation, the study of control and stabilization problems can be found in \cite{14,10, Zhang 1, Russell and Zhang, Rosier 1, Coron Crepau, Menzala Vasconcellos Zuazua, Rosier and Zhang 2}. Also, the Benjamin-Ono (BO) equation  has called the attention in the last decade (see \cite{1,Laurent Linares and Rosier, Linares Rosier} and the references therein). The Benjamin equation displays both a third order local term $\partial_{x}^{3}u$ as in the KdV equation, and a second order nonlocal term $\alpha \mathcal{H} \partial_{x}^{2}u$ as in the BO equation. So, it is natural to analyse the Benjamin equation from the control and stabilization point of view and check whether it  behaves in the similar way as the KdV and the BO equations.

Inspired by the recent works of Linares and  Ortega  \cite{1},   Russell and Zhang  \cite{10}, and Laurent, Rosier, and Zhang \cite{14} who respectively  studied the controllability and stabilization of the linearized BO equation and  the KdV equation on a periodic domain, we have obtained similar results for the linearized Benjamin equation as well. Different nature of eigenvalues  for the associated operator creates an obstacle in our case which we overcome using a generalized Ingham's inequality (see Remarks \ref{increasing} and \ref{otempo1} below).

Initially, we consider the initial value problem (IVP) associated to  equation \eqref{2D-BO}
in the periodic setting,
\begin{equation}\label{introduc}
\left \{
\begin{array}{l l}
    \partial_{t}u-\alpha\mathcal{ H}\partial^{2}_{x}u-\partial^{3}_{x}u=0,&  t\in\mathbb{R},\;\;x\in \mathbb{T}\\
    u(x,0)=u_{0}(x), & x\in \mathbb{T},
\end{array}
\right.
\end{equation}
with initial data $u_{0}(x)$ in an adequate space. As in the real setting,  with appropriate boundary conditions the equation \eqref{introduc}
admits the following conserved quantity
$$\int_{0}^{2\pi}u(x,t)\;dx=\int_{0}^{2\pi}u_{0}(x)\;dx.$$

The  IVP associated to  equation \eqref{2D-BO1}
in the periodic settings, can be written as
\begin{equation}\label{introduc2}
\left \{
\begin{array}{l l}
   \partial_{t}u-\alpha \mathcal{H}\partial^{2}_{x}u-\partial^{3}_{x}u=f(x,t),&  t\in (0,T),\;\;x\in \mathbb{T}\\
    u(x,0)=u_{0}(x), & x\in \mathbb{T},
\end{array}
\right.
\end{equation}
with initial data $u_{0}$ in an adequate space. The solution $u$ of system \eqref{introduc2}  satisfies
$$\frac{d}{dt}\left[\int_{0}^{2\pi}u(x,t)\;dx\right] =\int_{0}^{2\pi}f(x,t)\;dx.$$

So, the mass in the control system \eqref{introduc2} is indeed conserved  if we demand the function $f$ to satisfy
\begin{equation}\label{mediaf}
\int_{0}^{2\pi}f(x,t)\;dx= 0.
\end{equation}

In this work, the control function $f$ in \eqref{2D-BO1}   is allowed to act on only a small subset of the domain $\mathbb{T}$, i.e., $f$  is considered to be supported in a given open set $\omega\subset \mathbb{T}$. This situation includes more cases of practical interest and is therefore more relevant in general. With these considerations, we consider $g(x)$ as a real non-negative smooth function defined on $\mathbb{T}$
such that
\begin{equation}\label{gcondition}
2\pi[g]:=\int_{0}^{2\pi}g(x)\;dx=1,
\end{equation}
where $[g]$ represents the mean value of the function g over the interval $(0,2\pi).$ We assume
$\text{supp} \;g=\omega \subset \mathbb{T},$ where
$\omega=\{x\in \mathbb{T}: g(x)>0 \}$ is an open interval.
We will restrict our attention to control functions  of the form
\begin{equation}\label{EQ1}
f(x,t)=G(h)(x,t):=g(x)\left[h(x,t)-\int_{0}^{2\pi}g(y)
h(y,t)\;dy\right],\;\forall x\in \mathbb{T},\;t\in [0,T],
\end{equation}
where $h$ is a function defined in $\mathbb{T}\times [0,T].$
 Thus, $h\equiv h(x,t)$ can be considered as a new control function. Moreover,
for each $t\in [0,T] $ we have
that \eqref{mediaf} is satisfied.

Now we state the main results of this work which provide affirmative answers to the both questions posed above. The first main result deals with the controllability and reads as follows.
\begin{thm}\label{ControlLa}
	Let $s\geq 0,$ $\alpha>0,$ and  $T>0$ be given. 
	Then for each $u_{0},\; u_{1}\in H_{p}^{s}(\mathbb{T})$ with $[u_{0}]=[u_{1}],$ there exists a function $h\in L^{2}([0,T];H_{p}^{s}(\mathbb{T}))$
	such that the unique solution $u\in C([0,T];H^{s}_{p}(\mathbb{T}))$ of the non homogeneous system \eqref{introduc2} with $f(x,t)=G(h)(x,t)$
	satisfies $u(x,T)=u_{1}(x), \;x\in\mathbb{T}.$ Moreover, there exists a positive constant $\nu \equiv \nu(s,g,T)> 0$
	such that
$$\|h\|_{L^{2}([0,T];H_{p}^{s}(0,2\pi))} \leq \nu\; (\|u_{0}\|_{H_{p}^{s}(0,2\pi)}
	+\|u_{1}\|_{H_{p}^{s}(0,2\pi)}).$$
\end{thm}

\begin{rem}\label{increasing}
	The difficulty in the proof of
	Theorem \ref{ControlLa} comes from the fact that the sequence of
	eigenvalues associated to the Benjamin equation is not increasing,
	contrary to the case of the KdV and Benjamin-Ono equations  (see  Figure \ref{Autovalores1} below). The increasing property of the
	eigenvalues is a necessary condition to apply the Ingham's Theorem (see Theorem \ref{Ingham1}). Due to this reason,  we followed an approach implemented by
	Micu,  Ortega,  Rosier and Zhang in \cite{Micu Ortega Rosier and Zhang}
	and used a generalized form of the Ingham's
	inequality.
\end{rem}
\begin{figure}[h!]
	\begin{center}
		\includegraphics[width=13cm]{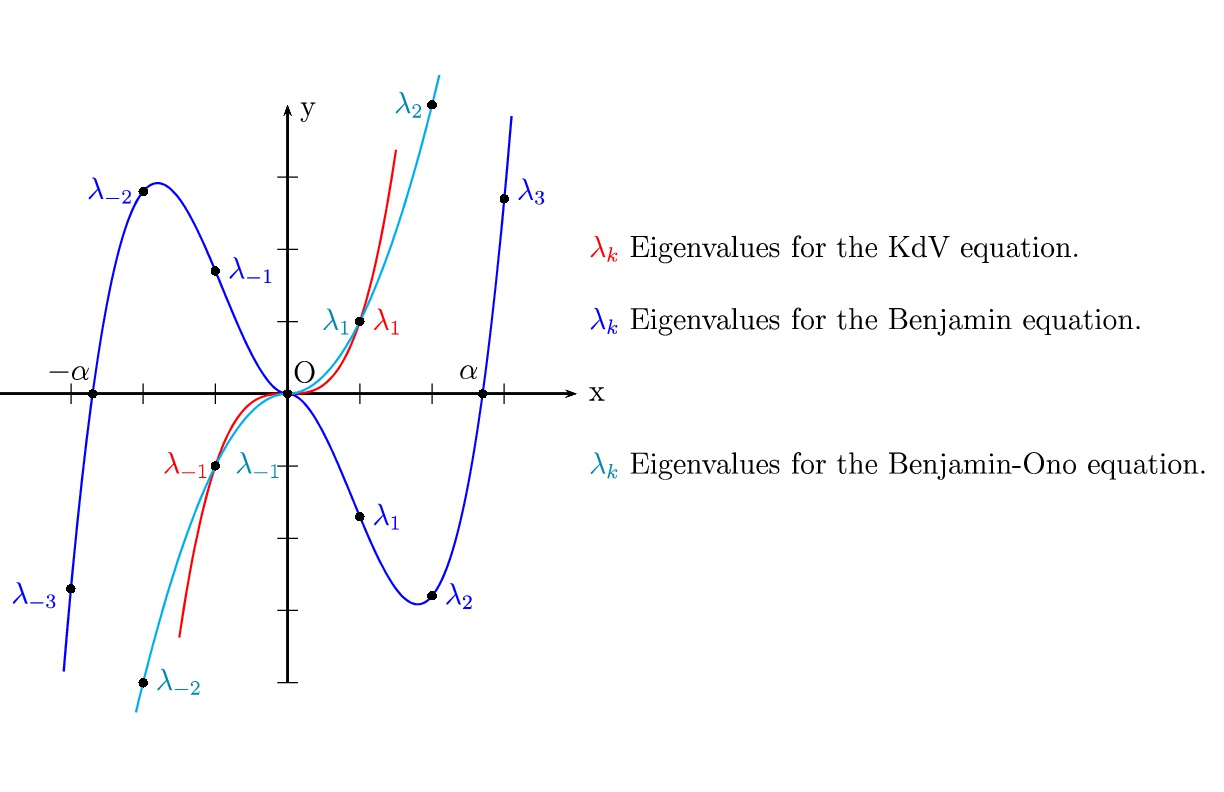}
		\caption{Eigenvalues}\label{Autovalores1}
	\end{center}
\end{figure}

\begin{rem}\label{otempo1}
	Theorem \ref{ControlLa} is strong from the point of view that we do not make restrictions neither on the eigenvalues of
	the operator $A\varphi=\alpha H\partial_{x}^{2}\varphi+\partial_{x}^{3}\varphi$  nor
	on the  time $T.$ 
	It is important to point out that the so called \textquotedblleft asymptotic gap condition" (see condition $iii)$ of Remark \ref{comporteigen} below) that holds for the eigenvalues associated to Benjamin equation was crucial to obtain the exact controllability for any positive time $T.$
\end{rem}

Regarding stabilization, we prove the following results.
\begin{thm}\label{st351}
	Let $\alpha>0,$  $g$ as in \eqref{gcondition},  and  $s\geq 0$ be given. There exist positive constants $M=M(\alpha, g, s)$ and $\gamma=\gamma(g),$ such that for any
	$u_{0}\in H_{p}^{s}(\mathbb{T})$  the unique solution $u\in C([0,\infty);H_{p}^{s}(\mathbb{T}))$
	of the closed-loop system \eqref{2D-BO2} with $Ku=-GG^{\ast}u$ satisfies
$$\|u(\cdot,t)-[u_{0}]\|_{H_{p}^{s}(\mathbb{T})}\leq M
	e^{-\gamma t}\|u_{0}-[u_{0}]\|_{H_{p}^{s}(\mathbb{T})},\;\;\;\text{for all}\;\;
	t\geq 0.$$
\end{thm}

Furthermore, using an observability inequality derived from the exact controllability result we can prove that the exponential decay rate of the resulting closed-loop system \eqref{2D-BO2} is as large as one desires. This is stated in the following theorem.

\begin{thm}\label{estabilization}
	Let $s\geq 0,$ $\alpha>0,$ $\lambda>0,$ and $u_{0}\in H_{p}^{s}(\mathbb{T})$ be given. There exists a bounded linear  operator $K_{\lambda}$  from $H_{p}^{s}(\mathbb{T})$ to $H_{p}^{s}(\mathbb{T})$ such that
	the unique solution $u\in C([0,+\infty), H_{p}^{s}(\mathbb{T}))$ of the closed-loop system 
	\eqref{2D-BO2} with $Ku=K_{\lambda}u$
	satisfies
	$$\|u(\cdot,t)-[u_{0}]\|_{H_{p}^{s}(\mathbb{T})}\leq
	M\;e^{-\lambda\;t}\|u_{0}-[u_{0}]\|_{H_{p}^{s}(\mathbb{T})},$$
for all $t\geq0,$ and some positive constant $M=M(g,\lambda, \alpha, s).$
\end{thm}

This theorem implies that for any given number $\lambda>0$ we can design a linear feedback control law such that the exponential decay rate of the resulting closed-loop system is $\lambda.$

 The paper is organized  as follows:
In section \ref{preliminares}  we list notations and a series of preliminary results which are used throughout this work.
 In Section \ref{section4}  we prove well-posedness results. The main  results regarding controllability and    stabilization are respectively proved  in Sections \ref{section5} and  \ref{section6}. Finally, in Section \ref{conc-rem} some concluding remarks and future works are presented.

\section{Preliminaries}\label{preliminares}

In this section we introduce some definitions, notations, properties and results related with Periodic Distributions, Sobolev spaces, and the Hilbert transform. We also introduce  Riesz basis, its properties and Ingham's inequality.

 We denote by $C^{\infty}_{p}(\mathbb{T})$  the space of all functions defined
on $\mathbb{T}$ that are infinitely differentiable and by $C_{p}(\mathbb{T})$ the space of all functions
defined on $\mathbb{T}$ that are continuous. 
We denote by $\mathcal{D}'(\mathbb{T})$ the space of all periodic distributions which is the dual space of $C^{\infty}_{p}(\mathbb{T}).$

\subsection{\textbf{Sobolev Spaces of $L^{2}$ type}}

Here we will introduce some definitions and results that involve Sobolev spaces. 
\begin{defn}
Let $s\in \mathbb{R},$ the Sobolev space of order $s$ on torus is defined by
$$H^{s}_{p}(\mathbb{T})=\left\{f\in \mathcal{D}'(\mathbb{T})\;/\;\|f\|_{H^{s}_{p}(\mathbb{T})}^{2}:=2\pi\sum_{k=-\infty}^{\infty}
(1+|k|^{2})^{s}|\widehat{f}(k)|^{2}<\infty \right\},$$
\end{defn}
where $\widehat{f}(k)$ is the $k^{th}-$Fourier coefficient of $f$ given by  
$$\widehat{f}(k)=\frac{1}{2\pi} \int_{0}^{2\pi}f(x)e^{-ikx}dx,\;\;\;\forall\;k\in\mathbb{Z}.$$

For all $s \in \mathbb{R},$  $ H^{s}_{p}(\mathbb{T}) $ is a Hilbert space
with the inner product
$$\displaystyle{(f\;, \;g)_{H^{s}_{p}(\mathbb{T})}= 2\pi\sum_{k\in \mathbb{Z}}
(1+|k|^{2})^{s}\widehat{f}(k)\;\overline{\widehat{g}(k)}<\infty.}$$
If $s=0$ then $ H^{0}_{p}(\mathbb{T}) $ is isometrically isomorphic to $L_{p}^{2}(\mathbb{T}).$
Moreover, given $s,r\in\mathbb{R}$ with $s\geq r$ one has
$$ H^{s}_{p}(\mathbb{T})\hookrightarrow H^{r}_{p}(\mathbb{T})$$ and this immersion is dense. We define a Fourier's orthonormal basis $\{ \psi_{k} \}_{k\in \mathbb{Z}}$ for $L_{p}^{2}(\mathbb{T})$ by
\begin{equation}\label{spi}
\psi_{k}(x):=\frac{e^{ikx}}{\sqrt{2 \pi}},\;\;\forall k\in \mathbb{Z},\;\;x\in \mathbb{T}.
\end{equation}

The  following Remark recalls a  characterization for Sobolev spaces. 
\begin{rem}\label{baseH1}
For $ s\geq 0,$  it is known that  $v\in H^{s}_{p}(\mathbb{T}) $ if and only if, for
$v(x)=\sum\limits_{k \in \mathbb{Z}}v_{k}\psi_{k}(x)$
we have that $\sum\limits_{k \in \mathbb{Z}}(1+|k|)^{2s}|v_{k}|^{2}<\infty.$
\end{rem}

\subsection{\textbf{The Hilbert transform (see \cite[page 66]{2})}}
Recall that   the Hilbert transform $\mathcal{H}$ defined by \eqref{hilberttramsf} can also be written as 
\begin{equation}\label{HilbertFrequences}
\widehat{\mathcal{H}(f)}(k)= -i\;\text{sgn}(k)\widehat{f}(k),\;\;\;\forall k\in\mathbb{Z}.
\end{equation}

The Hilbert transform  is  an isometry in $H_{p}^{s}(\mathbb{T})$
(see \cite[page 210]{6}) and satisfies the following properties.

\begin{prop}[The Hilbert Transform Properties]\label{ptf}
Assume  $f,\;g\in L_{p}^{2}(\mathbb{T}),$ then
\begin{equation}\label{Hprop1}
  \int_{\mathbb{T}}f(x)\;\overline{g}(x)\;dx=
  \int_{\mathbb{T}}\mathcal{H}(f)(x)\;\overline{\mathcal{H}(g)(x)}\;dx,
\end{equation}
\begin{equation}\label{Hprop2}
  \int_{\mathbb{T}}f(x)\;\overline{\mathcal{H}(g)(x)}\;dx=
  -\int_{\mathbb{T}}\mathcal{H}(f)(x)\;\overline{g(x)}\;dx,
\end{equation}
\begin{equation}\label{Hprop3}
\mathcal{H}\left(f\cdot \mathcal{H}(g)+\mathcal{H}(f)\cdot g\right)(y)=\mathcal{H}(f)(y)\cdot \mathcal{H}(g)(y)-f(y)\cdot g(y),
\end{equation}
\begin{equation}\label{Hprop4}
  \mathcal{H}(f)(x)=-i\sum_{k\in \mathbb{Z}}
  \text{sgn}(k)\widehat{f}(k)\;e^{ikx}.
\end{equation}

\end{prop}

\begin{proof}
To prove \eqref{Hprop1} and \eqref{Hprop2}, we use the Parseval's identity. The proof of \eqref{Hprop3} can be found in page 80 \cite{2}, and \eqref{Hprop4} is a direct consequense of \eqref{HilbertFrequences}. 
\end{proof}

\subsection{Riesz basis }\label{s5}
In  this subsection we record some definitions and results related to Riesz basis. Most of these can be found in Heil \cite{9}. In what follows, $J$ represents a
countable set of indices which could be finite or infinite.

\begin{defn}[{\cite[page 21]{9}}]\label{complet}
	Let $\{x_{n}\}_{n\in J}$ be a sequence in a normed linear space $X$. The finite linear span,
	or simply the span of $\{x_{n}\}_{n\in J}$  is the set of all finite linear combinations of
	elements of  $\{x_{n}\}_{n\in J}$
	$$\text{span} \left\{\{x_{n}\}_{n\in J}\right\}=
	\left\{\sum_{n=-N}^{N}c_{n}x_{n}:\;\text{for all} \;N> 0 \;\;\text{and}\;\; c_{1},...,c_{n}\in \mathrm{F}\right\}.$$
	
	We say that
	$\{x_{n}\}_{n\in J}$ is complete in $X$ if $\overline{\text{span}\left\{ \{x_{n}\}_{n\in J}\right\}}=X$
\end{defn}

\begin{defn}\label{Briesz}
	Let $\{x_{n}\}_{n\in J}$ be a sequence in a Hilbert space $X.$
	\begin{description}
		\item[i) Riesz basis] $\{x_{n}\}_{n\in J}$ is a Riesz basis if it is equivalent to some (and therefore every)
		orthonormal basis for $X.$
		
		\item[ii) Bessel sequence] A sequence $\{x_{n}\}_{n\in J}$ in a Hilbert space $X$ is a Bessel sequence if
		$$\forall x\in  X,\;\; \sum_{n \in J}|\langle x,x_{n}\rangle|^{2}<\infty.$$
		
	\end{description}
	
\end{defn}

\begin{defn}\label{Biorthogsystem}
	Given a Banach space $X$ and sequences $\{x_{n}\}_{n\in J}\subseteq X$ and
	$\{a_{n}\}_{n\in J}\subseteq X^{\ast},$ we say that $\{a_{n}\}$ is biorthogonal to
	$\{x_{n}\}$ if $\langle x_{m},a_{n}\rangle=\delta_{nm}$ for every $n,m\in J.$
	We call $\{a_{n}\}$ a biorthogonal system or a dual system of $\{x_{n}\}.$
\end{defn}

\begin{thm}[{\cite[page 197]{9}}]\label{BasRieszTheo} 
	Let $\{x_{n}\}_{n\in J}$ be a sequence in a Hilbert space $X.$ Then the following statements are equivalent.
	
	\begin{enumerate}
		\item $\{x_{n}\}_{n\in J}$ is a Riesz basis for $X.$
		
		\item $\{x_{n}\}_{n\in J}$ is a basis for $X,$ and
		$$\sum_{n \in J}c_{n}x_{n}\;\;\text{converges}\;\;\Leftrightarrow\;\;\sum_{n \in J}|c_{n}|^{2}\;\;\text{converges}.$$
		
		\item $\{x_{n}\}_{n\in J}$ is complete in $X$ and there exist constants $A,\;B>0$ such that
		$$\text{for all}\;\;c_{1},...,c_{N}\;\;\text{scalars,}\;\;A\sum_{n=1}^{N}|c_{n}|^{2}\leq \|\sum_{n=1}^{N}c_{n}x_{n}\|^{2}_{X}
		\leq B\sum_{n=1}^{N}|c_{n}|^{2}.$$
		
		\item $\{x_{n}\}_{n\in J}$ is a complete Bessel sequence and possesses a biorthogonal system $\{y_{n}\}_{n\in J}$ that is also
		a complete Bessel sequence.
	\end{enumerate}
\end{thm}

\begin{defn}\label{Minimalseq}
	We say that a sequence $\{x_{n}\}_{n\in J}$ in a Banach space $X$ is minimal if no vector $x_{m}$
	lies in the closed span of the other vectors $x_{n},$ it means,
	$$\forall m \in J,\;\;x_{m}\notin \overline{\text{span}\{\{x_{n}\}_{n\in J,\;n\neq m}\}}.$$
	
	A sequence that is both minimal and complete is said to be exact.
\end{defn}

\begin{lem}[{\cite[page 155]{9}}]\label{carctseqmini}
	Let $\{x_{n}\}_{n\in J}$ be a sequence in a Banach space $X.$ Then
	\begin{itemize}
		\item[1.]  there exists $\{a_{n}\}_{n\in J}\subseteq X^{\ast}\;\;\text{biorthogonal to}\;\;
		\{x_{n}\}_{n\in J}\;\Leftrightarrow\;\{x_{n}\}_{n\in J}\;\;\text{ is minimal}.$
		
		\item[2.] there exists a unique $\{a_{n}\}_{n\in J}\subseteq X^{\ast}\;\text{biorthogonal to}\;
		\{x_{n}\}_{n\in J}\;\Leftrightarrow\;\{x_{n}\}_{n\in J}\;\text{ is exact}.$
	\end{itemize}
	\end{lem}

\begin{thm}[{\cite[page 171]{9}}]\label{biort}
	If $\{x_{n}\}$ is a basis for a reflexive Banach space $X,$ then its biorthogonal system $\{a_{n}\}$ is a basis
	for $X^{\ast}.$
\end{thm}


\subsection{The Ingham's inequality}\label{s12}

Here we introduce the main tool to prove the controllability result for the linearized Benjamin equation, viz; the Ingham's inequality which is a generalization of Parseval's equality
due to Ingham \cite{8}. Further generalizations can be found in Komornik and Loreti  \cite{7} or in  Ball and Slemrod  \cite{Ball and Slemrod} and the references therein.

\begin{thm}[{\cite{8}}]\label{Ingham1}
	Let $\{\lambda_{k}\}_{k = -\infty}^{\infty}$ be a strictly increasing sequence of  real numbers, and
	$I$ a bounded interval. Consider the sums of the form $$f(t)=\sum_{k \in Z}c_{k}e^{i\lambda_{k}t},\;\;t\in I,$$
	with  square-summable complex coefficients $c_{k}.$
	Assume that there exists $\gamma>0$ such that the ``gap condition"
	$$\lambda_{n+1}-\lambda_{n} \geq \gamma, \;\;\forall\;n\in \mathbb{Z},$$
	holds, then there exist constants $A, \;B>0,$ such that for every bounded interval $I$
	of length  $|I|> \frac{2\pi}{\gamma},$
	$$A\sum_{k \in Z}|c_{k}|^{2}\leq \int_{I}|f(t)|^{2}dt\leq B\sum_{k \in Z}|c_{k}|^{2}.$$
\end{thm}

The following result is a generalization of  Theorem \ref{Ingham1}, for details see Theorem 4.6 in {\cite[page 67]{7}}

\begin{thm}[{\cite[page 67]{7}}]\label{InghamG2}
	Let $\{\lambda_{k}\}_{k\in J}$ be a family of real numbers, satisfying the uniform gap condition
	$$\gamma=	\underset{k\neq n}{\underset{k,n \in J}{\inf}} |\lambda_{k}-\lambda_{n}|>0.$$
	Set 
	$$\gamma'=\underset{S\subset J}{\sup}\;
	\underset{k\neq n}{\underset{k,n \in J\backslash S}{\inf}}
	|\lambda_{k}-\lambda_{n}|>0,$$	
	where $S$ rums over the finite subsets of $J.$
	
	If $I$ is a bounded interval of length $|I|> \frac{2\pi}{\gamma'},$ then there exists positive constants $A$ and $B$ such that
	$$A\sum_{k \in J}|c_{k}|^{2}\leq \int_{I}|f(t)|^{2}dt\leq B\sum_{k \in J}|c_{k}|^{2},$$
	for all functions given by the sum  $f(t)=\sum\limits_{k \in J}c_{k}e^{i\lambda_{k}t}$ with  square-summable complex coefficients $c_{k}.$
\end{thm}

\section{Well-posedness of the Linearized Benjamin equation}
\label{section4}
In this section we give some  properties of  the operator $G$ defined in \eqref{EQ1} and well-posedness results for the IVPs  \eqref{introduc} and  \eqref{introduc2}.

\subsection{\textbf{Properties of the operator}} We begin with  following property of $G$ which can be found in  \cite{1} (Remark 2.1) and \cite{Micu Ortega Rosier and Zhang} (Lemma 2.20).
\begin{prop}\label{Ghop}
Let $s\in \mathbb{R}.$ The operator $G:L^{2}\left([0,T]; H^{s}_{p}(\mathbb{T})\right)
\rightarrow L^{2}\left([0,T]; H^{s}_{p}(\mathbb{T})\right)$
is linear and bounded.
\end{prop}

\begin{prop}
The operator $G:L^{2}(\mathbb{T})\rightarrow L^{2}(\mathbb{T})$
is linear, bounded and self-adjoint.
\end{prop}
\begin{proof}
	It is easy to see that $G\in \mathcal{L}(L^{2}(\mathbb{T})).$ Moreover, there is a constant $C_{g}$ depending only on $g$ (see \eqref{gcondition}) such that
	$$\displaystyle{\|G(\varphi)\|_{L^{2}(\mathbb{T})}\leq C_{g}\|\varphi\|_{L^{2}(\mathbb{T})}}.$$
	
	We show that
	$G$ is symmetric. Let $h\in L^{2}(\mathbb{T}),$ thus
\begin{align*}
		(G(h)\;,\;f)_{L^{2}(\mathbb{T})}
		 & =\int_{0}^{2\pi}g(x)h(x)\overline{f}(x)\;dx
		 -\int_{0}^{2\pi}g(x)\overline{f}(x)\left[
		\int_{0}^{2\pi}g(y)
		h(y)\;dy\right]\;dx.
		\\ & =\int_{0}^{2\pi}g(y)h(y)\overline{f}
		(y)\;dy
	 -\int_{0}^{2\pi}g(y)h(y)\left[
		\int_{0}^{2\pi}g(x)
		\overline{f}(x)\;dx\right]\;dy
		\\ & =\int_{0}^{2\pi}g(y)
		h(y)\left[\overline{f}(y)-\int_{0}^{2\pi}
		g(x)\overline{f}(x)\;dx\right]\;dy
		\\ &=\int_{0}^{2\pi}h(y)G\left(\overline{f}\right)(y)\;
		dy\\
		 & =(h\;,\;G(f))_{L^{2}(\mathbb{T})} .
\end{align*}
	
	This proves the proposition.
\end{proof}

\subsection{\textbf{Well-posedness}}
In this subsection, we  establish global well-posedness for the linear IVP \eqref{introduc} and well-posedness for the non homogeneous system \eqref{introduc2} with $f=G(h).$ 

\begin{prop}\label{OGU}
Let $\alpha>0$ be given. The operator $A:D(A)\subseteq L^{2}(\mathbb{T})\rightarrow L^{2}(\mathbb{T}),$
defined by \;  $A\varphi:=\alpha \mathcal{H}\partial_{x}^{2}\varphi+\partial_{x}^{3}\varphi,$\;
generates a strongly continuous unitary group $\{U(t)\}_{t\in \mathbb{R}}$
on $L^{2}(\mathbb{T}).$
\end{prop}

\begin{proof}
Let $\varphi, \psi \in D(A)=H^{3}_{p}(\mathbb{T}). $
Using properties of the Hilbert transform we have
\begin{equation}\label{Hilbert1}
\mathcal{H}(\partial_{x}\varphi)(x)=\partial_{x}\mathcal{H}(\varphi)(x)\;\;\;\;\text{and}\;\;\;\;
\mathcal{H}(\partial_{x}^{2}\varphi)(x)=\partial_{x}^{2}\mathcal{H}(\varphi)(x),\;\forall x\in \mathbb{T},
\end{equation}
and
\begin{equation}\label{prop1}
( A\varphi\;,\; \psi)_{L^{2}(\mathbb{T})}= -\alpha \int_{0}^{2\pi}\partial_{x}^{2}\varphi(x) \overline{\mathcal{H}\psi(x)}\;dx+
\int_{0}^{2\pi}\partial_{x}^{3}\varphi(x)\overline{\psi(x)}\;dx.
\end{equation}

 Using integration by parts with respect to x in \eqref{prop1} we get
\begin{align*}
( A\varphi, \psi)_{L^{2}(\mathbb{T})}
&=-\int_{0}^{2\pi}\varphi(x)[\alpha\overline{\mathcal{H}\partial_{x}^{2}\psi(x)+\partial_{x}^{3}\psi(x)}]
\;dx\\
&=-(\varphi\;,\; A\psi)_{L^{2}(\mathbb{T})}.
\end{align*}
This implies that A is skew-adjoint and in particular  $( A\varphi\;,\; \varphi)_{L^{2}(\mathbb{T})}=0.$
Therefore, Theorem 3.2.3 in  Cazenave-Haraux \cite{3} implies that
the operator $A$ generates a strongly continuous unitary group of isometries (contractions) $\{U(t)\}_{t\in\mathbb{R}}$ on  $L^{2}(\mathbb{T}).$
\end{proof}

As a consequence of this proposition and Theorem 3.2.3 in  Cazenave-Haraux \cite{3} we have the following  global well-posedness  result for the IVP \eqref{introduc}  in $H_{p}^{3}(\mathbb{T})$.

\begin{cor}\label{EU}
Let $u_{0}\in H_{p}^{3}(\mathbb{T}),$ then there exists a unique solution
$$u\in C(\mathbb{R},H_{p}^{3}(\mathbb{T}))\cap C^{1}(\mathbb{R},L^{2}(\mathbb{T}))$$
for the homogeneous system \eqref{introduc}.
\end{cor}

We can generalize the last Corollary to get solutions
of the system  \eqref{introduc} in $H_{p}^{s}(\mathbb{T})$ for all $s\in \mathbb{R}.$
This, can be stated in a formal way as following.

Taking Fourier's transform in the spatial variable, the IVP \eqref{introduc} is equivalent to the following ODE
\begin{equation}\label{ODE1}
\left \{
\begin{array}{l l}
    \partial_{t}\widehat{u}(k)=ik^{2}\left[\alpha\;\text{sgn}(k)-k\right]\widehat{u}(k),&  t\in\mathbb{R},\\
    \widehat{u}(k,0)=\widehat{u_{0}}(k), &
\end{array}
\right.
\end{equation}
for all $ k\in \mathbb{Z}.$ The unique solution of  \eqref{ODE1} is given by
\begin{equation}\label{prop2}
\widehat{u}(t)(k)=e^{ik^{2}[\alpha\;
\text{sgn}(k)-k]t}\widehat{u_{0}}(k),\;\;\forall k\in \mathbb{Z}.
\end{equation}

Taking inverse Fourier transform in \eqref{prop2},  we get
\begin{equation}\label{solf}
 u(t)=\left(e^{ik^{2}[\alpha\;\text{sgn}(k)-k]t}\widehat{u_{0}}(k)\right)^{\vee},\;\;\forall\;t\in \mathbb{R}.
\end{equation}

It means that,
\begin{equation}\label{solf2}
 u(x,t)=\sum_{k\in \mathbb{Z}}
e^{ik^{2}[\alpha\;\text{sgn}(k)-k]t}\widehat{u_{0}}(k)e^{ikx},\;\;\forall\;t\in \mathbb{R},
\end{equation}
is the unique solution for the IVP  \eqref{introduc}.

Now, in a rigorous way, define the family of operators 
\;$U:\mathbb{R}\rightarrow \mathcal{L}(H^{s}_{p}(\mathbb{T}))$ by
\begin{equation}\label{semi2}
\begin{split}
 t\rightarrow U(t)\varphi:=e^{(\alpha H\partial_{x}^{2}+\partial_{x}^{3})t}\varphi
 &=(e^{ik^{2}[\alpha\;\text{sgn}(k)-k]t}\widehat{\varphi}(k))^{\vee}.
\end{split}
\end{equation}

Note that, with this definition the relation \eqref{solf} becomes $u(t)=U(t)u_{0},\;t\in \mathbb{R},$ and we get the following lemmas, whose proof can be obtained 
 from classical results on the semigroup theory (see for eg. Cazenave and Haraux \cite{3},  Pazy \cite{4} or  Iorio   and Magalhes  \cite{6} for more details).

\begin{lem}\label{semig2}
Let $s\in\mathbb{R}.$ The family of operators $\{U(t)\}_{t\in\mathbb{R}}$ given by \eqref{semi2}
defines a strongly continuous one-parameter unitary  group of contractions on $H^{s}_{p}(\mathbb{T}).$ Furthermore, $U(t)$ is an isometry for all $t\in\mathbb{R}.$
\end{lem}
\begin{proof} Here we only show that
$\displaystyle{\lim_{t\rightarrow t_{0}}\left\|U(t)f-U(t_{0})f\right\|_{H_{p}^{s}(\mathbb{T})}=0,\;\forall\;t_{0}\in \mathbb{R},}$ and $f\in H_{p}^{s}(\mathbb{T}).$
In fact, assume $t_{0}\in \mathbb{R}$, then
\begin{align*}
	\left\|U(t)f-U(t_{0})f\right\|_{H_{p}^{s}(\mathbb{T})}^{2}
	&= 2\pi\sum_{k\in \mathbb{Z}}(1+|k|^{2})^{s}\left|(e^{ik^{2}[\alpha\;\text{sgn}(k)-k]t}-e^{ik^{2}[\alpha\;\text{sgn}(k)-k]t_{0}}
	)\widehat{f}(k)\right|^{2}.
\end{align*}
	
	Note that $(1+|k|^{2})^{s}\left|(e^{ik^{2}[\alpha\;\text{sgn}(k)-k]t}-e^{ik^{2}[\alpha\;\text{sgn}(k)-k]t_{0}}
	)\widehat{f}(k)\right|^{2}\leq 4(1+|k|^{2})^{s}|\widehat{f}(k)|^{2},$
	and $$\sum_{k\in\mathbb{Z}}4(1+|k|^{2})^{s}|\widehat{f}(k)|^{2}\leq 4\|f\|_{H_{p}^{s}(\mathbb{T})}<\infty.$$
	Thus, a direct application of Weierstrass's M-test implies that the series
	$$\sum_{k\in \mathbb{Z}}(1+|k|^{2})^{s}\left|\left(e^{ik^{2}[\alpha\;Sgn(k)-k]t}-e^{ik^{2}[\alpha\;Sgn(k)-k]t_{0}}
	\right)\widehat{f}(k)\right|^{2}$$
	converges absolutely and uniformly with respect to $t.$ Therefore,
	$$\|U(t)f-U(t_{0})f\|_{H_{p}^{s}(\mathbb{T})}^{2}\rightarrow 0,\;\;\text{as}\;t\rightarrow t_{0}.$$
\end{proof}

\begin{lem}\label{semig3}
Assume $s\in\mathbb{R}.$ If $u(t)=U(t)u_{0},$ then
$$\lim_{h\rightarrow 0}\left\|\frac{u(t+h)-u(t)}{h}-[\alpha H\partial_{x}^{2}
+\partial_{x}^{3}]u\right\|_{H_{p}^{s-3}(\mathbb{T})}=0,$$
uniformly with respect $t\in\mathbb{R}.$
\end{lem}

Next theorem is a direct consequence of  Lemmas \ref{semig2} and  \ref{semig3}.

\begin{thm}\label{EU1}
	Let $s\in\mathbb{R}$ and $u_{0}\in H_{p}^{s}(\mathbb{T}),$ then there exists a unique solution
	$u\in C(\mathbb{R},H_{p}^{s}(\mathbb{T}))$
	for the homogeneous IVP \eqref{introduc}.
\end{thm}

In the following, we are going to deal with the well-posedness of the non-homogeneous system \eqref{introduc2} with $f=G(h)$ associated to the linearized Benjamin equation.
\begin{lem}\label{EUNH}
Let $0\leq T<\infty,$ $s\geq 0,$  $u_{0}\in H_{p}^{s}(\mathbb{T}),$ and $h\in L^{2}([0,T];H_{p}^{s}(\mathbb{T})).$
Then, there exists a unique mild solution
$u\in C([0,T],H_{p}^{s}(\mathbb{T})) $ for the IVP \eqref{introduc2} with $f=G(h).$
\end{lem}

\begin{proof}
For $h\in L^{2}([0,T];H_{p}^{s}(\mathbb{T})),$  the Proposition \ref{Ghop} implies $G(h)\in L^{2}([0,T];H_{p}^{s}(\mathbb{T})).$
Thus, $G(h)\in L^{1}([0,T];H_{p}^{s}(\mathbb{T})).$ We
rewrite the IVP \eqref{introduc2} with $f=G(h)$ in its equivalent form,
$$
\left \{
\begin{array}{l l}
    u\in C([0,T],H^{s}_{p}(\mathbb{T}))& \\
    \partial_{t}u=\alpha H\partial^{2}_{x}u+\partial^{3}_{x}u+Gh(t)\;\in H_{p}^{s-3}(\mathbb{T}),&  t\in(0,T)\\
    u(0)=u_{0}, &
\end{array}
\right.$$
where the initial data $u_{0}\in H^{s}_{p}(\mathbb{T}).$
From 
Corollary 2.2 and Definition 2.3
in Pazy \cite{4}, we have that
$$u(t)=U(t)u_{0}+\int_{0}^{t}U(t-t')Gh(t')dt'$$
is the unique solution of \eqref{introduc2} with $f=G(h)$ for $s\geq 0,\;0\leq t \leq T<\infty.$
\end{proof}

\section{Control of the linear Benjamin equation}
\label{section5}
In this section we prove an exact controllability result for the system \eqref{introduc2} with $f=G(h)$
 using the classical moment method, see \cite{Russell}.
Without loss of generality, one can consider  $u_{0}=0.$
In fact,  for given \;$u_{0},$ $u_{1}$ $\in H^{s}_{p}(\mathbb{T})$
with $[u_{0}]=[u_{1}],$ if $h$ is the control which leads
the solution $v$ of system \eqref{introduc2} with $f=G(h)$ from initial data $v_{0}=0$ to the final state $u_{1}-U(T)u_{0},$ then
$v$ can be written as,
$$v(t)=\int_{0}^{t}U(t-s)Gh(s)ds.$$

So 
$$ u_{1}-U(T)u_{0}=v(T)=\int_{0}^{T}U(T-s)Gh(s)ds.$$

Therefore, $$u_{1}=U(T)u_{0}+\int_{0}^{T}U(T-s)Gh(s)ds=u(T),$$
where $u$ is the solution of system \eqref{introduc2}  with $f=G(h)$ and initial data $u_{0}.$ It means that, the control $h$ leads the solution $u$ of system \eqref{introduc2} with $f=G(h)$ from the initial state $u_{0}$ to the final state $u_{1}.$

From this point onward we assume $u_{0}=0,$ so that $[u_{1}]=[u_{0}]=0.$ In consequence, we have $c_{0}=0$ whenever we write $u_{1}(x)=\sum\limits_{k\in \mathbb{Z}}c_{k}\;\psi_{k}(x)\in H^{s}_{p}(\mathbb{T}),$ with $s\geq0,$ and $\psi_{k}$ as  in 
\eqref{spi}.
The next result is fundamental to get control for the linear system \eqref{introduc2} with $f=G(h).$

\begin{lem}\label{caractc}
Let $ s\geq 0,$ and $T>0$ be given. Assume $u_{1}\in H^{s}_{p}(\mathbb{T})$ with $[u_{1}]=0.$
Then, there exists
$h\in L^{2}([0,T],H^{s}_{p}(\mathbb{T})),$ such that the solution of the IVP \eqref{introduc2} with $f=G(h)$ and initial data $u_{0}=0$ satisfies $u(T)=u_{1}$ if and only if
\begin{equation}\label{CEQ}
    \int_{0}^{T}\left\langle Gh(\cdot,t),\varphi(\cdot,t)\right\rangle_{H^{s}_{p}\times(H^{s}_{p})'}dt
  =\left\langle u_{1},\varphi_{0}\right\rangle_{H^{s}_{p}\times(H^{s}_{p})'},
\end{equation}
for any $\varphi_{0}\in (H^{s}_{p}(\mathbb{T}))'$, where $(H^{s}_{p}(\mathbb{T}))'$ is the dual space of $H^{s}_{p}(\mathbb{T}),$ and $\varphi$
is the solution of the adjoint system 
\begin{equation}\label{adsis}
\left \{
\begin{array}{l l}
    \partial_{t}\varphi-\alpha H\partial^{2}_{x}\varphi-\partial^{3}_{x}\varphi=0,&  t>0,\;\;x\in \mathbb{T}\\
    \varphi(x,T)=\varphi_{0}(x), & x\in \mathbb{T}.
\end{array}
\right.
\end{equation}
\end{lem}
\begin{proof}
$(\Rightarrow)$ Let $\varphi_{0}$ and $h$ be smooth functions and $\varphi$ be the solution of the adjoint system
\eqref{adsis} with final data $\varphi_{0}.$ Multiplying the equation in \eqref{introduc2} by $\overline{\varphi},$   integrating by parts, and using the Hilbert transform proprieties in Proposition \ref{ptf}, we obtain
\begin{equation}\label{c2}
\small{
\begin{split}
\int_{0}^{T}\int_{0}^{2\pi} Gh\; \overline{\varphi}\;dx\;dt 
 &=\int_{0}^{2\pi}u(T)\;\overline{\varphi}(T)\;dx
 -\int_{0}^{T}\int_{0}^{2\pi} u\;\left[\partial_{t}\overline{\varphi}- \alpha\partial^{2}_{x}\mathcal{H}\overline{\varphi}-
 \partial^{3}_{x}\overline{\varphi}\right]\;dx\;dt\\
 &=\int_{0}^{2\pi} u(T)\;\overline{\varphi}(T)\;dx.
\end{split}}
\end{equation}

Therefore,
\begin{align*}
 \int_{0}^{T}\int_{0}^{2\pi} Gh\; \overline{\varphi}\;dx\;dt&=\int_{0}^{2\pi} u_{1}\;\overline{\varphi_{0}}\;dx.
\end{align*}
Now, identifying $L^{2}(\mathbb{T})$ with its dual (see  \cite[page 254]{Zeidler}) by means of the (conjugate linear) map
$y \rightarrow (\cdot,y)_{L^{2}(\mathbb{T})},$  we have the following inclusion
 $$H_{p}^{s}(\mathbb{T})\hookrightarrow L^{2}(\mathbb{T})
  \equiv (L^{2}(\mathbb{T}))'\hookrightarrow (H_{p}^{s}(\mathbb{T}))',$$
where the embedding is dense and continuous. Moreover,
$\left\langle \phi, \varphi\right\rangle_{H^{s}_{p}\times(H^{s}_{p})'}=
( \phi, \varphi)_{L^{2}(\mathbb{T})},$ for all $\phi,$ $\varphi\in L^{2}(\mathbb{T}).$ Thus,
\begin{align*}
\int_{0}^{T}\left\langle Gh(\cdot,t), \varphi(\cdot,t)\right\rangle_{H^{s}_{p}\times(H^{s}_{p})'}\;dt
&=\int_{0}^{T}( Gh(\cdot,t), \varphi(\cdot,t))_{L^{2}(\mathbb{T})}\;dt\\
&=\int_{0}^{2\pi} u_{1}\;\overline{\varphi_{0}}\;dx\\
&=\left\langle u_{1},\; \varphi_{0}\right\rangle_{H^{s}_{p}\times(H^{s}_{p})'}.
\end{align*}
$(\Leftarrow)$ Let $h$ be a smooth function such that \eqref{CEQ} holds for any smooth $\varphi_{0}\in (H^{s}_{p}(\mathbb{T}))'$.
Identifying $L^{2}(\mathbb{T})$ with its dual and using \eqref{c2}, we have
\begin{equation}\label{c3}
\begin{split}
 \int_{0}^{2\pi} u_{1}\;\overline{\varphi_{0}}\;dx
 &= \left\langle u_{1},\; \varphi_{0}\right\rangle_{H^{s}_{p}\times(H^{s}_{p})'}\\
 &=\int_{0}^{T}\left\langle Gh(\cdot,t), \varphi(\cdot,t)\right\rangle_{H^{s}_{p}\times(H^{s}_{p})'}\;dt\\
&= \int_{0}^{T}\int_{0}^{2\pi} Gh\; \overline{\varphi}\;dx\;dt\\
&=\int_{0}^{2\pi} u(T)\;\overline{\varphi}(T)\;dx\\
&=\int_{0}^{2\pi} u(T)\;\overline{\varphi_{0}}\;dx.
\end{split}
\end{equation}

Identity \eqref{c2} implies that
$\displaystyle{\int_{0}^{2\pi} [u(T)-u_{1}]\;\overline{\varphi_{0}}\;dx=0,
 \; \text{for all smooth function}\; \varphi_{0}.}$
In consequence $u(T)=u_{1}.$
Thus, the lemma is true for all smooth data

In general case, we use density arguments to complete the proof.
\end{proof}

The following result is a characterization for the existence of control to the  system \eqref{introduc2} with $f=G(h)$ and initial data $u_{0}=0.$

\begin{lem}\label{coef}
Let $ s\geq 0,$ $T>0$ be given, and $ \psi_{k}(x) $ as in \eqref{spi}.
If $$u_{1}(x)=\sum_{l\in \mathbb{Z}}
c_{l}\psi_{l}(x)\;\;\in H^{s}_{p}(\mathbb{T}),$$ is a function such that $[u_{1}]=0,$
then  the  system \eqref{introduc2} with $f=G(h)$ and initial data $u_{0}=0$
is exactly controllable to $u_{1},$ that is, $u(x,T)=u_{1}(x), \;\;\forall x\in\mathbb{T},$
if an only if there exists $h\in L^{2}([0,T];H^{s}_{p}(\mathbb{T}))$ such that
\begin{equation}\label{caract5}
\int_{0}^{T}\int_{\mathbb{T}}Gh(x,t)\;
\;e^{-i\lambda_{k}(T-t)}\overline{\psi_{k}(x)}\;dx\;dt=c_{k},\;\forall\;k\in \mathbb{Z},
\end{equation}
where $\lambda_{k}:=k^{3}-\alpha k|k|.$

\end{lem}

\begin{proof}
$(\Rightarrow)$ In view of Lemma \ref{caractc}, let us to consider the adjoint system
\begin{equation}\label{adja}
\left \{
\begin{array}{l l}
    \partial_{t}\varphi-\alpha H\partial^{2}_{x}\varphi-\partial^{3}_{x}\varphi=0,&  t>0,\;\;x\in \mathbb{T}\\
    \varphi(x,T)=\varphi_{0}(x), & x\in \mathbb{T}
\end{array}
\right.
\end{equation}
and let $k\in \mathbb{Z}$ be fixed.
 Note that
$\psi_{k}\in (H^{s}_{p}(\mathbb{T}))'.$ So, we suppose $\varphi_{0}=\psi_{k}.$
Then identity \eqref{solf} implies that
\begin{equation}\label{sol1}
\begin{split}
\varphi(x,t)&=U(t-T)\varphi_{0}(x)\\
&=U(T-t)^{\ast}\varphi_{0}(x)\\
&=\left(e^{-il^{2}[\alpha\;\text{sgn}(l)-l](T-t)}\widehat{\varphi_{0}}(l)\right)^{\vee}\\
&=\sum_{l\in \mathbb{Z}}
e^{-il^{2}[\alpha\;\text{sgn}(l)-l](T-t)}\widehat{\varphi_{0}}(l)e^{ilx}\\
&=\sum_{l\in \mathbb{Z}}
e^{i\lambda_{l}(T-t)}\widehat{\varphi_{0}}(l)e^{ilx},
\end{split}
\end{equation}
where $\lambda_{l}=l^{3}-\alpha l\;|l|.$ Since
	\begin{align*}
	\widehat{\psi_{k}}(l)&=\left \{
	\begin{array}{l l}
		\frac{1}{\sqrt{2\pi}}, & if\; k=l \\
		0, &  if \;k\neq l.
	\end{array}
	\right.
\end{align*}
we obtain from \eqref{sol1} that
\begin{align*}
\varphi(x,t)
&=e^{i\lambda_{k}(T-t)}\psi_{k}(x).
\end{align*}

Now, using identity \eqref{CEQ} one gets
\begin{align*}
 \int_{0}^{T}\int_{\mathbb{T}}Gh(x,t)\;\overline{\varphi}(x)\;\;dx\;dt&
 -\int_{\mathbb{T}}
 \left[\sum_{l\in \mathbb{Z}}
c_{l}\psi_{l}(x)\right]\;\overline{\varphi_{0}}(x)\;dx=0.
\end{align*}

Therefore,
\begin{align*}
 \int_{0}^{T}\int_{\mathbb{T}}Gh(x,t)\;e^{-i\lambda_{k}(T-t)}\psi_{-k}(x)\;
 \;dx\;dt&=\int_{\mathbb{T}}
 \left[\sum_{l\in \mathbb{Z}}
c_{l}\psi_{l}(x)\right]\;\psi_{-k}(x)\;dx\\
&= \sum_{l\in \mathbb{Z}}
c_{l}\int_{\mathbb{T}}
\psi_{l}(x)\;\psi_{-k}(x)\;dx\\
 &=c_{k}\;\;\forall k\in \mathbb{Z},
\end{align*}
as required.\\

\noindent
$(\Leftarrow)$
Now, suppose that there exists $h\in L^{2}([0,T];H^{s}_{p}(\mathbb{T}))$ such that
\eqref{caract5} holds. With similar calculations as above, we  obtain
$$\int_{0}^{T}\int_{\mathbb{T}}Gh(x,t)\;
\;\overline{e^{i\lambda_{k}(T-t)}\;\psi_{k}}(x)\;dx\;dt-\int_{0}^{2\pi} u_{1}\;\overline{\psi_{k}}\;dx=0,\;
\forall\;\varphi_{0}=\psi_{k}, \;k\in \mathbb{Z}.$$

Multiplying both sides of the last equality by $\overline{\widehat{\varphi_{0}}(k)}$
and summing over $k\in \mathbb{Z}$, we get
\begin{align*}
\sum_{k\in \mathbb{Z}}\int_{0}^{T}\int_{\mathbb{T}}Gh(x,t)\;
\;\overline{e^{i\lambda_{k}(T-t)}\;\psi_{k}}(x)\;\overline{\widehat{\varphi_{0}}(k)}dx\;dt&=
\sum_{k\in \mathbb{Z}}\int_{0}^{2\pi} u_{1}(x)\;\overline{\psi_{k}}(x)\;
\overline{\widehat{\varphi_{0}}(k)}\;dx.
\end{align*}

Note that
\begin{align*}
\varphi(x,t)&=\sum_{k\in \mathbb{Z}}
e^{i\lambda_{k}(T-t)}\widehat{\varphi_{0}}(k)e^{ikx}
\end{align*}
is the solution of the adjoint system \eqref{adja} and  $\varphi_{0}\in C_{p}^{\infty}(\mathbb{T})$ can be expressed as
$$\varphi_{0}(x)=\sum_{k\in \mathbb{Z}}\widehat{\varphi_{0}}(k)\;\psi_{k}(x),$$
where the series converge uniformly. Thus
 $$\int_{0}^{T}\int_{\mathbb{T}}Gh(x,t)\;
\;\overline{\varphi}(x,t)\;dx\;dt-\int_{0}^{2\pi} u_{1}\;\overline{\varphi_{0}}\;dx=0,\;
\forall\;\varphi_{0}\in C_{p}^{\infty}(\mathbb{T}).$$

The result follows by using  density arguments.
\end{proof}

\begin{lem}\label{invertmatrix}
Let $ \psi_{k}(x) $ be as in \eqref{spi}, and
\begin{equation}\label{invertmatrix1}
m_{j,k}=\widehat{G(\psi_{j})}(k)=\int_{0}^{2\pi}G(\psi_{j})(x) \overline{\psi_{k}}(x)\;dx,\;\;\;\;j,k\in\mathbb{Z},
\end{equation}
where $G$ is as in \eqref{EQ1}. In addition, for any given finite sequence of nonzero integers $k_{j}$, j=1,2,3,....,n,
 let
$$M_{n}= \left(
   \begin{array}{ccc}
     m_{k_{1},k_{1}} & \cdots & m_{k_{1},k_{n}}  \\
     m_{k_{2},k_{1}} & \cdots & m_{k_{2},k_{n}} \\
     \vdots & \vdots & \vdots \\
     m_{k_{n},k_{1}} & \cdots & m_{k_{n},k_{n}} \\
   \end{array}
 \right).$$

 Then

 \begin{itemize}
   \item [$i)$]  there exists a constant $\beta>0,$ depending only on $g$, such that
   $$m_{k,k}\geq \beta,\;\;\;\text{ for any}\; k\in\mathbb{Z}-\{0\}.$$

    \item[$ii)$] $m_{j,0}=0,\;\;\;\text{ for any}\; j\in\mathbb{Z}.$

   \item [$iii)$]$M_{n}$ is an invertible $n\times n$ hermitian matrix.

   \item [$iv)$] there exists $\delta>0,$ depending only on $g$, such that
\begin{equation}\label{posit1}
    \delta_{k}=\|G(\psi_{k})\|^{2}_{L^{2}(\mathbb{T})}> \delta >0,\;\;\text{for all}\;k\in \mathbb{Z}-\{0\}.
\end{equation}

 \end{itemize}

\end{lem}

\begin{proof}
The proof of items $i)$, $ii)$, and $iii)$ can be found in  \cite[page 296]{Micu Ortega Rosier and Zhang}.
The proof of $iv)$ can be found in \cite[page 3650]{10} (see also  \cite[page 213]{1}).
\end{proof}

\begin{rem}\label{comporteigen}
The sequence of eigenvalues $\{\lambda_{k}\}_{k \in \mathbb{Z}}$, with
$\lambda_{k}=k^{3}-\alpha k|k|$, satisfies the following properties:
\begin{itemize}
  \item[$i)$] $\lambda_{-k}=-\lambda_{k}$, for all $k\in \mathbb{Z}$.

  \item[$ii)$] $\displaystyle{\lim_{|k| \rightarrow \infty}|\lambda_{k}|=\infty}$.

  \item[$iii)$] $\displaystyle{\lim_{|k| \rightarrow \infty}|\lambda_{k+1}-\lambda_{k}|=\infty}$ (asymptotic gap condition).

  \item[$iv)$] Observe that not all the eigenvalues of the sequence $\{\lambda_{k}\}_{k \in \mathbb{Z}}$
  are distint, it depends on the value of $\alpha$. 
  For each $k_{1}\in \mathbb{Z}$ set
  $I(k_{1})=\{k\in \mathbb{Z}: \lambda_{k}=\lambda_{k_{1}}\}\;\;\;\text{and}\;\;\;|I(k_{1})|=m(k_{1}),$
where $|I(k_{1})|$ denotes the numbers of elements of $I(k_{1}).$
  Then we have the following properties for $m(k_{1}):$
  \begin{itemize}
    \item [a)] $m(k_{1})\leq 3,$ for all $k_{1}\in \mathbb{Z}$. This is a consequence of the fact that
    $m(k_{1})$ is less or equal to the number of integer roots of the equation $f(x):=x^{3}-\alpha x|x|=\beta$,
    where $\beta$ is an arbitrary real number, see the format of the curve in Figure \ref{Autovalores} below.
\begin{figure}[h!]
	\begin{center}
		\includegraphics[width=13cm]{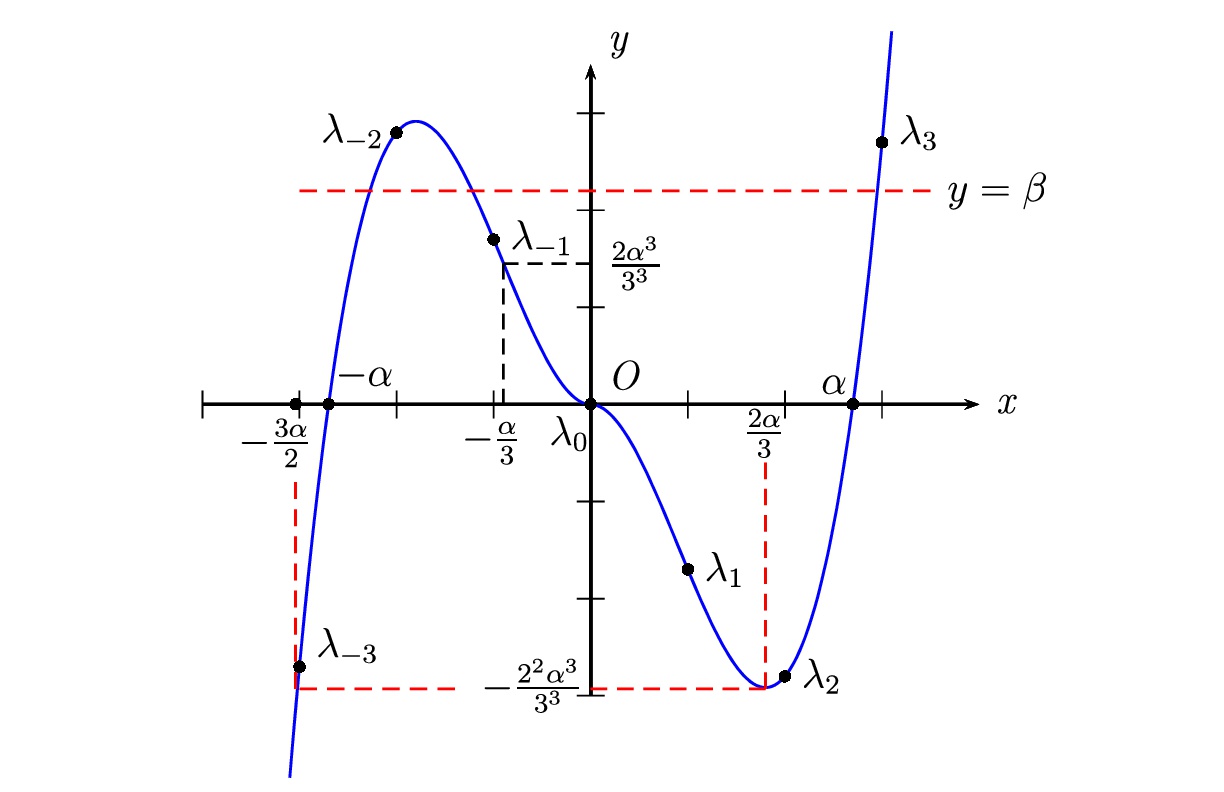}
		\caption{Eigenvalues}\label{Autovalores}
	\end{center}
\end{figure}

    \item [b)]If the sequence of eigenvalues tend to infinity, there exists $k_{1}^{\ast}\in \mathbb{N}$
    such that $m(k_{1})=1$, for all $|k_{1}|>k_{1}^{\ast}$. This is a consequence of the fact
    that the function $x\rightarrow x^{3}-\alpha x |x|$ is strictly increasing
    for $|x|$ large enough.
      \end{itemize}
  
\item[$v)$]  If we count only the distinct eigenvalues, we obtain a sequence
$\{\lambda_{k}\}_{k \in \mathbb{I}}$, where $\mathbb{I}\subseteq\mathbb{Z}$ has the
property that $\lambda_{k_{1}}\neq \lambda_{k_{2}}$, for any $k_{1},k_{2}\in \mathbb{I}$, with $k_{1}\neq k_{2}$.

 \item[$vi)$] From part $a)$ in $iv)$ we infer that there are only
 finitely many integers in $\mathbb{I},$ say, $k_{j},$ $j=1,2,3,...,n,$
 such that one can find another integer $k\neq k_{j}$
 with $\lambda_{k}=\lambda_{k_{j}}.$ Let
 $$\mathbb{I}_{j}:=\{k\in \mathbb{Z}: k\neq k_{j}, \lambda_{k}=\lambda_{k_{j}}\},\;\;\;\;j=1,2,3,...,n.$$
 
 Then $$\mathbb{Z}=\mathbb{I}\cup\mathbb{I}_{1}\cup\mathbb{I}_{2}\cup\cdots\cup\mathbb{I}_{n},$$
where the sets in the right are pairwise disjoint.
  
 \item[$vii$)] From part $b)$ in $iv),$ we infer that 
  \begin{equation}\label{gamma}
  \begin{split}
  \gamma&:=\inf_{\substack{k,n\in \mathbb{I}\\k\neq n}}|\lambda_{k}-\lambda_{n}|
  = \min_{F}|\lambda_{k}-\lambda_{n}|>0,
  \end{split}
  \end{equation}
  where $F:=\left\{n,k\in\mathbb{I}: k\neq n, \;\text{and}\;\;-1-[\frac{3\alpha}{2}]\leq k,n \leq [\frac{3\alpha}{2}]+1\right\}$,
 because\\ $x\rightarrow x^{3}-\alpha x |x|$ is increasing very fast for $|x|>[\frac{3\alpha}{2}]+1.$ 
\end{itemize}
\end{rem}

Now we provide proof of our main theorem regarding controllability of non-homogeneous linear system \eqref{introduc2} with $f=G(h),$ stated in Theorem \ref{ControlLa}.

\begin{proof}[Proof of Theorem \ref{ControlLa}] As discussed above, it is enough to consider $u_{0}=0.$ We prove this theorem in five steps. 
	
\noindent
{\bf Step 1.} We  show that the family $\{e^{i\lambda_{k}t}\}_{k\in \mathbb{I}}$ is a Riesz basis
	(see Definition \ref{Briesz}) for the
closed span $\overline{\text{span}\{e^{i\lambda_{k}t}: k\in \mathbb{I}\}}=:H$ in $L^{2}([0,T]),$ where the set of indices $\mathbb{I}$ was defined in part $v)$ of Remark \ref{comporteigen}.

In fact, since  $L^{2}([0,T])$ is a reflexive separable  Hilbert space so is $H$.
It follows from Definition \ref{complet}  that the sequence  $\{e^{i\lambda_{k}t}\}_{k\in \mathbb{I}}$
is complete in $H.$
On the other hand,  from  item $iii)$ of Remark \ref{comporteigen},  the eigenvalues associated to the linearized Benjamin equation satisfy the assymptotic gap condition which implies
	$$\gamma'=\underset{S\subset \mathbb{I}}{\sup}\;
\underset{k\neq n}{\underset{k,n \in \mathbb{I}\backslash S}{\inf}}
|\lambda_{k}-\lambda_{n}|=+\infty,$$	
where $S$ rums over the finite subsets of $\mathbb{I}.$
Using  Theorem \ref{InghamG2} with $\gamma$ defined by 
\eqref{gamma}, we obtain 
 that there exist positive constants $A$ and $B,$  such that
\begin{equation}\label{indes}
 A\sum_{n \in \mathbb{I}}|b_{n}|^{2}\leq \int_{0}^{T}|f(t)|^{2}dt\leq B\sum_{n \in \mathbb{I}}|b_{n}|^{2},
\end{equation}
for all functions of the form  $f(t)=\sum\limits_{ n \in \mathbb{I}}b_{n}e^{i\lambda_{n}t},$ $t\in [0,T]$
with  square-summable complex coefficients $b_{n}.$ In particular, if $b_{1},...,b_{N}$ are $N$ arbitrary  constants we have
$$\displaystyle{A\sum_{n=1}^{N}|b_{n}|^{2}\leq \int_{0}^{T}\left|\sum_{n \in \mathbb{I}}b_{n}e^{i\lambda_{n}t}\right|^{2}dt
\leq B\sum_{n=1}^{N}|b_{n}|^{2}}.$$
Thus
 $$\displaystyle{A\sum_{n=1}^{N}|b_{n}|^{2}\leq \left\|\sum_{n \in \mathbb{I}}b_{n}e^{i\lambda_{n}t}\right\|^{2}_{H}
 \leq B\sum_{n=1}^{N}|b_{n}|^{2}\;\;\text{for all}\;\;b_{1},...,b_{N}\;\;\text{scalars}.}$$

Now, applying Theorem \ref{BasRieszTheo} we conclude that
$\{e^{i\lambda_{k}t}\}_{k\in \mathbb{I}}$ is a Riesz basis for the
closed span $H$ in $L^{2}([0,T]).$ 

\noindent
{\bf Step 2.}  In this step we show the existence of a
unique biorthogonal dual basis $\{q_{j}\}_{j\in \mathbb{I}}\subseteq H^{\ast}.$

 Indeed, theorem \ref{BasRieszTheo}  implies that
$\{e^{i\lambda_{k}t}\}_{k\in \mathbb{I}}$ is a complete Bessel sequence and possesses a biorthogonal system $\{q_{j}\}_{j\in \mathbb{I}}$ which is also
a complete Bessel sequence. Moreover, Theorem  \ref{biort}
 implies that $\{q_{j}\}_{j\in \mathbb{I}}$ is a basis for $H^{\ast}$ which can be identified with $H,$
therefore, $\{q_{j}\}_{j\in \mathbb{I}}$ is also a Riesz basis for $H.$
So, by Lemma \ref{carctseqmini} part $1,$ we get that $\{e^{i\lambda_{k}t}\}_{k\in \mathbb{I}}$ is minimal.
In consequence, we have the existence of a unique biorthogonal dual basis $\{q_{j}\}_{j\in \mathbb{I}}\subseteq H^{\ast}$
due to exactness (see Definition \eqref{Minimalseq})  of the sequence $\{e^{i\lambda_{k}t}\}_{k\in \mathbb{I}}$ and   Lemma \ref{carctseqmini} part $2.$
Thus
\begin{equation}\label{dualb}
    (e^{i\lambda_{k}t}\;,\;q_{j})_{H}=\int_{0}^{T}e^{i\lambda_{k}t}\overline{q_{j}}(t)\;dt=\delta_{kj},\;\;\forall\;k,j\in \mathbb{I}.
\end{equation}

\noindent
{\bf Step 3.}  Here we will define an adequate control function $h.$

In fact, in Step 2, we  found a sequence of functions $q_{j}$ where $j$ is running on the set of indices $\mathbb{I}.$  In this step, we will need to define a sequence of functions $q_{j}$ with $j$ running on  $\mathbb{Z}$. Note that, $\mathbb{Z}=\mathbb{I}\cup\mathbb{I}_{1}\cup\mathbb{I}_{2}\cup\cdots\cup\mathbb{I}_{n},$  so it is enough to define this sequence for indices in $\mathbb{I}_j$, $j=1,\cdots, n$. Furthermore, recall from  part $vi)$ in Remark \ref{comporteigen} that, each $\mathbb{I}_{j}$ contains at most
2 integers. Without loss of generality, we may assume that
$$\mathbb{I}_{j}=\{k_{j,1}, k_{j,2}\},\;\;\;\;j=1,2,3,\cdots,n.$$

We denote $k_{j}$ by $k_{j,0}$ for any $j=1,2,3,..,n.$ Therefore, for
$k_{j,l}$ we define
$$q_{k_{j,l}}:=q_{k_{j,0}}=q_{k_{j}},\;\;\;\text{for all}\;j=1,2,3,...,n,\;\;\text{and}\;\;l=0,1,2.$$

Also, it is important to note that $$\lambda_{k_{j,l}}=\lambda_{k_{j}}
,\;\;\;\text{for all}\;j=1,2,3,\cdots,n,\;\;\text{and}\;\;l=0,1,2.$$

For suitable $h_{j}$'s, consider a control function $h$ defined  by
\begin{equation}\label{thecontrol}
    h=\sum_{j \in \mathbb{Z}} h_{j}\;\overline{q_{j}}(t)\;\psi_{j}(x).
\end{equation}

 Note that, using the identity $G(\overline{q_{j}}(t)\;\psi_{j})=\overline{q_{j}}(t)\;G(\psi_{j}),$ we obtain
\begin{equation}\label{thecontrol1}
{\footnotesize
\begin{split}
\int_{0}^{T}\int_{0}^{2\pi}G(h)(x,t) e^{-i\lambda_{k}(T-t)} \overline{\psi_{k}}(x)\;dx\;dt&=
\int_{0}^{T}\int_{0}^{2\pi}\left(\sum_{j\in \mathbb{Z}}h_{j}\overline{q_{j}}(t)G(\psi_{j})(x,t)\right)
e^{-i\lambda_{k}(T-t)} \overline{\psi_{k}}(x)\;dx\;dt\\
&=
\sum_{j\in \mathbb{Z}}h_{j}\int_{0}^{T}\overline{q_{j}}(t) e^{-i\lambda_{k}(T-t)} \;dt\int_{0}^{2\pi}G(\psi_{j})(x) \overline{\psi_{j}}(x)\;
\;dx\\
&=\sum_{j\in \mathbb{Z}}h_{j} e^{-i\lambda_{k}T} m_{j,k} 
\int_{0}^{T}\overline{q_{j}}(t)e^{i\lambda_{k}t}\;dt.
\end{split}}
\end{equation}

\noindent
{\bf Step 4.} 
 In this step we find  $h_{j}$'s such that $h$ defined by \eqref{thecontrol} serves as a required control function. For this, we use the identity \eqref{thecontrol1} and  Lemma \ref{coef} applied to
$$u_{1}(x)=\sum_{n\in \mathbb{Z}}
c_{n}\psi_{n}(x)\;\;\in H^{s}_{p}(\mathbb{T}),\;\;\text{with}\;\;[u_{1}]=0\;\;(c_{0}=0\;\;\text{and}\;\;u_{0}=0),$$
to infer that it is enough to consider $h_{j}$'s satisfying
\begin{equation}\label{h form}
   c_{k}=\sum_{j\in \mathbb{Z}}h_{j}e^{-i\lambda_{k}T}m_{j,k}
\int_{0}^{T}\overline{q_{j}}(t)e^{i\lambda_{k}t}\;dt.
\end{equation}

Note that,  part $ii)$ of Lemma \ref{invertmatrix} implies that the equation \eqref{h form} is
satisfied for $k=0,$ independently of the values of $h_{j}.$  Moreover, from \eqref{dualb} we obtain that
$$c_{k}=h_{k}m_{k,k}\;e^{-i\lambda_{k}T},\;\;\;\text{if}\;k\neq k_{j,l},\;\;
l=0,1,2,\;\;j=1,2,3,...,n;$$
and for $k= k_{j,l},\;\;
l=0,1,2,\;\;j=1,2,3,...,n.$
$$\left\{
    \begin{array}{ll}
      \displaystyle{c_{k,0}=\sum_{l=0}^{2}h_{k_{j,l}} m_{k_{j,l},k_{j,0}}e^{-i\lambda_{k_{j,0}}T}} & \hbox{;} \\
     \displaystyle{c_{k,1}=\sum_{l=0}^{2}h_{k_{j,l}} m_{k_{j,l},k_{j,1}}e^{-i\lambda_{k_{j,1}}T}} & \hbox{;} \\
     \displaystyle{c_{k,2}=\sum_{l=0}^{2}h_{k_{j,l}} m_{k_{j,l},k_{j,2}}e^{-i\lambda_{k_{j,2}}T}} & \hbox{.}
    \end{array}
  \right.$$

 Therefore, choosing $h_{0}=0,$ and using part $iii)$ of  Lemma \ref{invertmatrix}, we obtain
\begin{equation}\label{hform3}
h_{k}=\frac{c_{k}\;\; e^{i\lambda_{k}T}}{m_{k,k}},\;\;\;\text{if}\;k\neq 0\;\text{and}\;k\neq k_{j,l},\;\;
l=0,1,2,\;\;j=1,2,3,...,n;
\end{equation}
and
\begin{equation}\label{hform4}
{\footnotesize
\left(
  \begin{array}{c}
    h_{k_{j,0}}\\
    h_{k_{j,1}}\\
    h_{k_{j,2}}\\
  \end{array}
\right)^{\top}=
\left(
  \begin{array}{c}
    c_{k_{j,0}} e^{i\lambda_{k_{j,0}}T}\\
    c_{k_{j,1}} e^{i\lambda_{k_{j,1}}T}\\
    c_{k_{j,2}} e^{i\lambda_{k_{j,2}}T}\\
  \end{array}
\right)^{\top}
M_{j}^{-1},\;\;\text{for} \; j=1,2,3,...,n,}
\end{equation}
where
$$M_{j}=\left(
  \begin{array}{ccc}
    m_{k_{j,0},k_{j,0}} & m_{k_{j,0},k_{j,1}} & m_{k_{j,0},k_{j,2}} \\
    m_{k_{j,1},k_{j,0}} & m_{k_{j,1},k_{j,1}} & m_{k_{j,1},k_{j,2}} \\
    m_{k_{j,2},k_{j,0}} & m_{k_{j,2},k_{j,1}} & m_{k_{j,2},k_{j,2}}. \\
  \end{array}
\right).$$

In this way, we take $h_{j}$'s given by \eqref{hform3} and 
\eqref{hform4}. \\

\noindent
{\bf Step 5.} 
In this step we prove that the unique function $h$ defined by
\eqref{thecontrol}
belongs to $L^{2}([0,T];H_{p}^{s}(\mathbb{T}),$ where 
$h_{0}=0,$ and $h_{k}$ with $k\neq 0$ is defined by \eqref{hform3} and \eqref{hform4}.

 Indeed, identifying $H^{\ast}$ with $H$, and using the Remark \ref{baseH1}, together with the fact
that $\{q_{j}\}_{j\in \mathbb{I}}$ is a Riesz basis for $H$ we obtain
\begin{align*}
\begin{split}
\|h\|^{2}_{L^{2}([0,T];H_{p}^{s}(\mathbb{T}))} 
&=C\sum_{k\in \mathbb{Z}} (1+|k|)^{2s} \int_{0}^{T}|h_{k}\; q_{k}(t)|^{2}\;dt\\
&\leq C \sum_{k\in \mathbb{Z}}(1+|k|)^{2s} B_{2}\;|h_{k}|^{2},
\end{split}
\end{align*}
where $B_{2}$ is the constant given by the Bessel type inequality (similar to \eqref{indes}) for
the Riesz basis $\{q_{j}\}_{j\in \mathbb{I}}$ in $H.$
Thus, from identity \eqref{hform3} and Lemma \eqref{invertmatrix} part $i),$  we obtain
\begin{equation}\label{hform5}
{\footnotesize
\begin{split}
\|h\|^{2}_{L^{2}([0,T];H_{p}^{s}(\mathbb{T}))}
&\leq C B_{2}\underset{l=0,1,2\;\;\;j=1,2,...,n}
{\sum_{ k\in \mathbb{Z}-\{0\}\;\;k\neq k_{j,l}}} (1+|k|)^{2s}\left|\frac{c_{k} e^{i\lambda_{k}T}}{m_{k,k}}\right|^{2}
+ C B_{2} \sum_{j=1}^{n}\sum_{l=0}^{2}(1+|k_{j,l}|)^{2s}|h_{k_{j,l}}|^{2} \\
&\leq \frac{C B_{2}}{\beta^{2}}\underset{l=0,1,2\;\;\;j=1,2,...,n}
{\sum_{ k\in \mathbb{Z}-\{0\}\;\;k\neq k_{j,l}}} (1+|k|)^{2s}\left|c_{k}\right|^{2}
+ C B_{2} \sum_{j=1}^{n}\sum_{l=0}^{2}(1+|k_{j,l}|)^{2s}|h_{k_{j,l}}|^{2}.
\end{split}}
\end{equation}

 From identity \eqref{hform4} we obtain that for each $l=0,1,2$ and $j=1,2,...,n$
\begin{align*}
\begin{split}
|h_{j,l}|^{2}&\leq \sum_{m=0}^{2}|h_{j,m}|^{2}
\leq \left( \sum_{m=0}^{2}\left|c_{k_{j,m}}e^{i\lambda_{k_{j,m}}}\right|^{2} \right)\|M_{j}^{-1}\|^{2}
\leq \|M_{j}^{-1}\|^{2}\sum_{m=0}^{2}|c_{k_{j,m}}|^{2},
\end{split}
\end{align*}
where $\|M_{j}^{-1}\|$ is the Euclidean norm of the matrix $M_{j}^{-1}.$
This implies that for each $l=0,1,2$ and $j=1,2,...,n$
\begin{equation}\label{hform7}
\begin{split}
(1+|k_{j,l}|)^{2s}|h_{j,l}|^{2}
& \leq\sum_{m=0}^{2}\|M_{j}^{-1}\|^{2}
\frac{(1+|k_{j,l}|)^{2s}}{(1+|k_{j,m}|)^{2s}}(1+|k_{j,m}|)^{2s}|c_{k_{j,m}}|^{2}\\
 &\leq C(s)\sum_{m=0}^{2}(1+|k_{j,m}|)^{2s}|c_{k_{j,m}}|^{2},
\end{split}
\end{equation}
where $\displaystyle{C(s)=\underset{m,l=0,1,2}{\max_{j=1,2,...,n}}
 \left\{\|M_{j}^{-1}\|^{2}\frac{(1+|k_{j,l}|)^{2s}}{(1+|k_{j,m}|)^{2s}}\right\} }.$
 
Therefore, using inequalities \eqref{hform5}, and \eqref{hform7}, we obtain
\begin{equation}\label{cota}
{\footnotesize
\begin{split}
\|h\|^{2}_{L^{2}([0,T];H_{p}^{s}(\mathbb{T}))}
&\leq \frac{ C B_{2}}{\beta^{2}}\underset{l=0,1,2\;\;\;j=1,2,...,n}
{\sum_{ k\in \mathbb{Z}-\{0\}\;\;k\neq k_{j,l}}} (1+|k|)^{2s}\left|c_{k}\right|^{2}
+ 3C B_{2} C(s)\sum_{j=1}^{n}
\sum_{m=0}^{2}(1+|k_{j,m}|)^{2s}|c_{k_{j,m}}|^{2}\\
& \leq \nu^{2} \|u_{1}\|^{2}_{H^{s}_{p}(\mathbb{T})},
\end{split}}
\end{equation}
where $\displaystyle{\nu^{2}\equiv \nu^{2}(s,g,T)=\max\left\{\frac{ C B_{2}}{\beta^{2}}, 3CB_{2}C(s)\right\}}$.

 This completes the proof of the theorem.
\end{proof}

\begin{rem}
The dependence of $\nu$ with respect to $T$ is implicit in the constant $B_{2}$ which is obtained by applying Theorem \ref{InghamG2}.
\end{rem}

 Theorem \ref{ControlLa} allows us to get the following corollary.

\begin{cor} \label{controloperator}
For $s\geq 0,$ and $T>0$ given, there exists a unique
bounded linear  operator
$\Phi:H_{p}^{s}(\mathbb{T})\times H_{p}^{s}(\mathbb{T})\rightarrow L^{2}([0,T];H_{p}^{s}(\mathbb{T}))$ defined by
$\Phi(u_{0},u_{1}):=h,$ for all $(u_{0},\;u_{1})\in H_{p}^{s}(\mathbb{T})\times H_{p}^{s}(\mathbb{T})$ (see \eqref{thecontrol})
such that
\begin{equation}\label{cont}
u_{1}=U(T)u_{0}+\int_{0}^{T}U(T-s)(G(\Phi(u_{0},u_{1})))(\cdot,s)\;ds,
\end{equation}
 and
\begin{equation}\label{oprestima}
\|\Phi(u_{0},u_{1})\|_{L^{2}([0,T];H_{p}^{s}(\mathbb{T}))} \leq \nu\; (\|u_{0}\|_{H_{p}^{s}(\mathbb{T})}
+\|u_{1}\|_{H_{p}^{s}(\mathbb{T})}),
\end{equation}
where $\nu$ depends only on $s,\;T,$ and $g$ (see \eqref{gcondition}).
\end{cor}

 Also, Lemma \ref{caractc} and Corollary \ref{controloperator} allow us to get the following observability inequality, which is fundamental to obtain a result on exponential asymptotic stabilization with decay rate as large as one desires for the system   \eqref{2D-BO2}.

 \begin{cor} \label{controloperator1}
  Let $T>0$ be given. There exists $\delta>0$ such that
$$\int_{0}^{T}\|GU(-\tau)\phi\|
    ^{2}_{L^{2}(\mathbb{T})}(\tau) \; d\tau \geq \delta^{2}
    \|\phi\|^{2}_{L^{2}(\mathbb{T})}, $$
    for any $\phi \in L^{2}(\mathbb{T}).$
 \end{cor}
\begin{proof}
	Let $T>0.$ Define a linear map $F_{T}:L^{2}([0,T]; L^{2}_{p}(\mathbb{T}))\rightarrow L^{2}_{p}(\mathbb{T})$ by
	\begin{equation}\label{caractc3}
		F_{T}(h)=u(\cdot,T),
	\end{equation}
	where $u=u(x,t)$ is the solution (mild solution) of
	\begin{equation}\label{caractc4}
	\left \{
	\begin{array}{l l}
	\partial_{t}u-\alpha H\partial^{2}_{x}u-\partial^{3}_{x}u=Gh(x,t),&  t\in (0,T),\;\;x\in \mathbb{T}\\
	u(x,0)=0, & x\in \mathbb{T}.
	\end{array}
	\right.
	\end{equation}	
	
	Note that if $u_{1}\in L^{2}_{p}(\mathbb{T})$ is given, then
	from Corollary \ref{controloperator} there exists
	$h$ such that 
	\begin{eqnarray}\label{ftprop}
	F_{T}(h)=u(T)=u_{1}.
	\end{eqnarray}

	Therefore, $F_{T}$ is onto and trivially $ \text{Ran}(F_{T})$ is dense in $L^{2}_{p}(\mathbb{T}).$
	
	On the other hand, from Corollary \ref{controloperator}, for $\;u_{1}\in L_{p}^{2}(\mathbb{T}),$ we have that
	\begin{equation}\label{cont1}
	u_{1}=\int_{0}^{T}U(T-s)(G(h))(\cdot,s)\;ds.
	\end{equation}
	
	Therefore, from \eqref{ftprop} and \eqref{cont1}
\begin{align*}
	\|F_{T}(h)\|_{L^{2}_{p}(\mathbb{T})}
	&=\left\|\int_{0}^{T}U(T-s)(G(h))(\cdot,s)\;ds
	\right\|_{L^{2}_{p}(\mathbb{T})}\\
	&\leq \int_{0}^{T}\left\|U(T-s)(G(h))(\cdot,s)
	\right\|_{L^{2}_{p}(\mathbb{T})}\;ds\\
	&\leq c_{g}\;\int_{0}^{T}\|h\|_{L^{2}_{p}(\mathbb{T})}\;ds\\
	&\leq c_{g}\;T^{\frac{1}{2}} \|h\|_{L^{2}_{p}([0,T];L^{2}(\mathbb{T}))}.
\end{align*}

	So, $F_{T}$ is a bounded linear operator. Thus, $F_{T}^{\ast}$ exists, is a bounded linear operator, and  is one-to-one (see  Rudin  \cite[Corollary b) page 99]{Rudin}).
	Also, from Theorem 4.13  in \cite{Rudin} (see also  \cite[page 35]{Coron}), we have that there exists
	$\delta >0$ such that
	\begin{equation}\label{caractc12}
	\left\|F_{T}^{\ast}(\phi^{\ast})
	\right\|_{\left(L^{2}([0,T];L^{2}_{p}(\mathbb{T}))\right)'}
	\geq \delta \;\|\phi^{\ast}\|_{\left(L^{2}_{p}(\mathbb{T})\right)'},
	\;\;\;\text{ for all}\;\;
	\phi^{\ast}\in \left(L^{2}_{p}(\mathbb{T})\right)'.
	\end{equation}

	From Lemma \ref{caractc},  we have that the solution $u$ of  \eqref{caractc4} satisfies
	\begin{equation}\label{caractc6}
	\int_{0}^{T}\left\langle Gh(\cdot,t),\varphi(\cdot,t)\right\rangle_{L^{2}_{p} \times(L^{2}_{p})'}dt
	-\left\langle u_{1},\varphi_{0}\right\rangle_{L^{2}_{p}\times(L^{2}_{p})'}=0,
	\end{equation}
	for any $\varphi_{0}\in (L^{2}_{p}(\mathbb{T}))',$ and $\varphi$
	the solution of  the adjoint system \eqref{adsis}. 
	Note that $$\varphi(\cdot,t)=U(T-t)^{\ast}\varphi_{0}.$$ Then it follows from
	\eqref{caractc6} that
	\begin{align*}
	\int_{0}^{T}\left\langle h(\cdot,t),G^{\ast}U(T-t)^{\ast}\varphi_{0}
	\right\rangle_{L^{2}_{p}(\mathbb{T})\times
		(L^{2}_{p}(\mathbb{T}))'}dt
	&=\left\langle u(\cdot,T),\varphi_{0}\right
	\rangle_{L^{2}_{p}(\mathbb{T})\times
		(L^{2}_{p}(\mathbb{T}))'}\\
	&=\left\langle F_{T}(h),\varphi_{0}\right
	\rangle_{L^{2}_{p}(\mathbb{T})\times
		(L^{2}_{p}(\mathbb{T}))'}\\
	&=\left\langle h\;,\;F_{T}^{\ast}\varphi_{0}\right\rangle
	_{L^{2}([0,T];L^{2}_{p}(\mathbb{T})) \times \left(L^{2}([0,T];L_{p}^{2}(\mathbb{T}))\right)'}.
	\end{align*}
	
	Therefore, $F_{T}^{\ast}=G^{\ast}U(T-t)^{\ast},$ and using \eqref{caractc12},
	we have
$$\left\|G^{\ast}U(T-t)^{\ast}(\phi^{\ast})
	\right\|_{L^{2}([0,T];(L^{2}_{p}(\mathbb{T}))')}
	\geq \delta \;\|\phi^{\ast}\|_{\left(L^{2}_{p}(\mathbb{T})\right)'},
	\;\;\;\text{ for all}\;\;
	\phi^{\ast}\in \left(L^{2}_{p}(\mathbb{T})\right)'.$$
	
	It means,
$$\int_{0}^{T}\|G^{\ast}U(T-t)^{\ast}(\phi^{\ast}(x))
	\|^{2}_{(L^{2}_{p}(\mathbb{T}))'}\;dt
	\geq \delta^{2}\; \|\phi^{\ast}\|^{2}_{(L_{p}^{2}(\mathbb{T}))'},\;\;\;\text{ for all}\;\;
	\phi^{\ast}\in (L^{2}_{p}(\mathbb{T}))'.$$

	Performing a change of the temporal variable $ \tau=T-t,$ we obtain 
$$\int_{0}^{T}\|G^{\ast}U(\tau)^{\ast}(\phi^{\ast}(x))
	\|^{2}_{(L^{2}_{p}(\mathbb{T}))'}\;dt
	\geq \delta^{2}\; \|\phi^{\ast}\|^{2}_{(L_{p}^{2}(\mathbb{T}))'},\;\;\;\text{ for all}\;\;
	\phi^{\ast}\in (L^{2}_{p}(\mathbb{T}))'.$$

Identifying $L^{2}(\mathbb{T})$ with its dual we conclude the proof.
\end{proof}

Before ending this section, we record an observation. To simplify the calculations in the study of the stabilization problem, we would like to consider solutions of system \eqref{introduc2} with  mean zero. Note that, in general, the assumption $[u(\cdot,t)]=[u_{0}]=[u_{1}]=0$   is not valid for solutions of system \eqref{introduc2}.
To solve that problem, let $u$ be a solution of equation \eqref{2D-BO1} with
$[u(\cdot,t)]=[u_{0}]=[u_{1}]=:\mu,$ for all $t\in [0,T]$  and
let $v(x,t)=u(x,t)-\mu.$ Note that $v$ solves
\begin{equation}\label{nonlinearintro3}
	\left \{
	\begin{array}{l l}
		\partial_{t}v-\alpha H\partial^{2}_{x}v-\partial^{3}_{x}v+2\mu\partial_{x}v
		=f(x,t),&  t\in(0,T),\;\;x\in \mathbb{T}\\
		v(x,0)=v_{0}(x):=u_{0}(x)-\mu, & x\in \mathbb{T},
	\end{array}
	\right.
\end{equation}
where $\mu \in \mathbb{R},$ and
$[v(\cdot,t)]=[v_{0}]=0,\;\;\forall t\in [0,T].$ 

Conversely, if $v$ is a solution of equation \eqref{nonlinearintro3}
then, $u(x,t)=v(x,t)+\mu$ is a solution of system \eqref{introduc2}. In consequence, we must resolve the 
controllability and stabilization problems for the system
\eqref{nonlinearintro3}.
As before, we begin by considering the linear non homogeneous  system 
\begin{equation}\label{linearintro3}
	\left \{
	\begin{array}{l l}
		\partial_{t}v-\alpha H\partial^{2}_{x}v-\partial^{3}_{x}v+2\mu\partial_{x}v=Gh(x,t),&  t\in(0,T),\;\;x\in \mathbb{T}\\
		v(x,0)=v_{0}(x), & x\in \mathbb{T},
	\end{array}
	\right.
\end{equation}
where  $v_{0}\in H_{p}^{s}(\mathbb{T})$ with $s\geq0.$
As the operator $A_{\mu}:D(A_{\mu})\subseteq L^{2}(\mathbb{T})\rightarrow L^{2}(\mathbb{T}),$
defined by
\begin{equation}\label{ope2}
A_{\mu}\varphi=\alpha H\partial_{x}^{2}\varphi+\partial_{x}^{3}\varphi-2\mu\partial_{x}\varphi
\end{equation}
is skew-adjoint, it
generates a strongly continuous unitary group $\{U_{\mu}(t)\}_{t\in \mathbb{R}}$
on $L^{2}(\mathbb{T}).$ Moreover, for $s\in\mathbb{R}$ the family of operators $\{U_{\mu}(t)\}_{t\in\mathbb{R}}$ given by
\begin{equation}\label{semgru2}
	\begin{split}
		U_{\mu}:\mathbb{R}\rightarrow \mathcal{L}(H^{s}_{p}(\mathbb{T}))\qquad \qquad \qquad \qquad \qquad &\\
		t\rightarrow U_{\mu}(t)\varphi:=e^{(\alpha H\partial_{x}^{2}+\partial_{x}^{3}-2\mu\partial_{x})t}\varphi
		=\left(e^{i(-k^{3}-2\mu k+\alpha k|k|)t}\widehat{\varphi}(k)\right)^{\vee}, &
	\end{split}
\end{equation}
defines a strongly continuous one-parameter unitary  group of contractions on $H^{s}_{p}(\mathbb{T}).$ Furthermore, $U_{\mu}(t)$ is an isometry for all $t\in\mathbb{R}.$

\begin{rem}\label{sol3}
	  For $s\in \mathbb{R}$ and $v_{0}\in H_{p}^{s}(\mathbb{T}),$  (respectively $v_{0}\in H_{p}^{3}(\mathbb{T})$)
	  we obtain
	that there exists a unique solution $v\in C(\mathbb{R},H_{p}^{s}(\mathbb{T}))$
	(respectively, $v\in C(\mathbb{R},H_{p}^{3}(\mathbb{T}))\cap C^{1}(\mathbb{R},L^{2}(\mathbb{T}))$)
	for the homogeneous equation associated to equation \eqref{linearintro3}. Furthermore, if
	$0\leq T<\infty,$ $s\geq 0,$  $v_{0}\in H_{p}^{s}(\mathbb{T}),$ and $h\in L^{2}([0,T];H_{p}^{s}(\mathbb{T}))$
	then, there exists a unique mild solution
	$v\in C([0,T],H_{p}^{s}(\mathbb{T})) $ for the system \eqref{linearintro3}.
\end{rem}

\begin{rem}\label{sol4}
	We get an analogous result of Lemma \ref{coef} for the system \eqref{linearintro3}, just modifying  $\lambda_{k}=k^{3}-\alpha k|k|$\;\; by \;\;
	$\lambda_{k}=k^{3}+2\mu k-\alpha k|k|.$ Also, due to the ``asymptotic gap condition"
	that holds for the eigenvalues of the operator $A_{\mu},$ we have an analogous result of Theorem \ref{ControlLa} for the equation \eqref{linearintro3}, it means that the system \eqref{linearintro3} is exactly controllable.

	Thus, similarly to Corollary \ref{controloperator}, for $s\geq 0$ and any $T>0$ given, there exists  a
	bounded linear  operator
	$$\Phi_{\mu}:H_{p}^{s}(\mathbb{T})\times H_{p}^{s}(\mathbb{T})\rightarrow L^{2}([0,T];H_{p}^{s}(\mathbb{T}))$$ defined by
	$h_{\mu}=\Phi_{\mu}(v_{0},v_{1}),$ for all  $(v_{0},v_{1})\in H_{p}^{s}(\mathbb{T})\times H_{p}^{s}(\mathbb{T})$
	such that
	\begin{equation}\label{contLab}
		v_{1}=U_{\mu}(T)v_{0}+\int_{0}^{T}U_{\mu}(T-s)(G(\Phi_{\mu}(v_{0},v_{1})))(\cdot,s)\;ds,
	\end{equation}
 and
	\begin{equation}\label{oprestimaLab}
		\|\Phi_{\mu}(v_{0},v_{1})\|_{L^{2}([0,T];H_{p}^{s}(\mathbb{T})} \leq \nu\; (\|v_{0}\|_{H_{p}^{s}(\mathbb{T})}
		+\|v_{1}\|_{H_{p}^{s}(\mathbb{T})}),
	\end{equation}
	where $\nu$ depends only on $s,\;T,$ and $g.$ Therefore, the following observability inequality holds
{\small
	\begin{equation}\label{CEQ6}
		\int_{0}^{T}\|G^{\ast}U_{\mu}(\tau)^{\ast}(\phi(x))\|^{2}_{L^{2}_{p}(\mathbb{T})}\;d\tau
		\geq \delta^{2}\; \|\phi\|^{2}_{L^{2}(\mathbb{T})},\;\;\;\text{ for any}\;\;
		\phi\in L^{2}(\mathbb{T}),\;\;\text{ some}\;\;\delta>0,
\end{equation} }
and for any $T>0.$
\end{rem}

\section{Stabilization of the linear Benjamin equation}\label{section6}
In this section we prove the exponential stabilization results stated in Theorem \ref{st351} and Theorem \ref{estabilization}. From the observation made in the final part of section \ref{section5}, it is enough to  study the stabilization problem for the linear IVP \eqref{nonlinearintro3} in $H_{0}^{s}(\mathbb{T})$ with $s\geq0,$ where 
$$H_{0}^{s}(\mathbb{T}):=\left\{u\in H_{p}^{s}(\mathbb{T}):
[u(\cdot,t)]=0,\;\;\text{for all}\;\;t>0\right\}.$$
If $s=0$ then, we denote $H_{0}^{0}(\mathbb{T})$ by $L_{0}^{2}(\mathbb{T}).$ Here,
we mention some properties of these Sobolev spaces.
\begin{prop}\label{stabilspace}
	$H_{0}^{s}(\mathbb{T})$ is a closed subspace of $H_{p}^{s}(\mathbb{T})$ for all
	$s\geq 0.$ In particular, $L_{0}^{2}(\mathbb{T})$ is a closed subspace of
	$L^{2}(\mathbb{T}).$
\end{prop}

\begin{rem}\label{st17}
	The Proposition \ref{stabilspace} implies that
	$(H_{0}^{s}(\mathbb{T}), \|\cdot\|_{H_{p}^{s}(\mathbb{T})})$ is a Hilbert space for all $s\geq 0$.
	Furthermore, it is easy to show that if $s\geq r \geq 0$ then
	$H_{0}^{s}(\mathbb{T}) \hookrightarrow H_{0}^{r}(\mathbb{T}),$ where the  embedding is dense.
\end{rem}

So, we study the stabilization problem for the system
\begin{equation}\label{atabilizationL2}
\left \{
\begin{array}{l l}
\partial_{t}u-\alpha \mathcal{H}\partial^{2}_{x}u-\partial^{3}_{x}u+2\mu\partial_{x} u=Ku,&  t>0,\;\;x\in \mathbb{T}\\
u(x,0)=u_{0}(x), & x\in \mathbb{T},
\end{array}
\right.
\end{equation}
where $u=u(x,t)$ is real valued function, $\alpha>0,$ and $K$ is a bounded linear operator on
$H^{s}_{0}(\mathbb{T})$. In view of the discussion at the end of the previous section we assume that  $\mu \in \mathbb{R},$ and
$[u(\cdot,t)]=0,$ for all $t\geq0.$ 

\subsection{Stabilization of the linear Benjamin equation}
In this subsection we prove that 
there exists a feedback control law such that
the system \eqref{atabilizationL2} is exponentially asymptotically stable when $t$ goes to infinity.
First, we prove that  the system \eqref{atabilizationL2} is globally well-posed in $H_{0}^{s}(\mathbb{T})$, $s\geq 0$.
\begin{thm}\label{solsta1}
	Let $u_{0}\in H_{0}^{3}(\mathbb{T}),$ then the IVP \eqref{atabilizationL2}
	has a unique solution
	$$u\in C([0,\infty);H_{0}^{3}(\mathbb{T}))
	\cap C^{1}([0,\infty);L^{2}_{0}(\mathbb{T})).$$
	
	Moreover, if $u_{0}\in H_{0}^{s}(\mathbb{T}),$ then we have that $u\in C([0,\infty);H_{0}^{s}(\mathbb{T})),$ for all $s \geq 0.$
\end{thm}

\begin{proof}
	We know that the operator
	$A_{\mu}=\alpha \mathcal{H}\partial_{x}^{2}+\partial_{x}^{3}-2\mu\partial_{x}$
	is an infinitesimal generator of a
	$C_{0}$-semigroup $\{U_{\mu}(t)\}_{t\geq 0}$ over $H_{0}^{s}(\mathbb{T}).$ Also we know that $K$ is a bounded linear operator on $H^{s}_{0}(\mathbb{T}).$
	From the semigroup theory
	(see pg. 76 in  \cite{4}),
	we get that the operator
	$A_{\mu}+K,$ which is a perturbation of $A_{\mu}$ by a bounded linear operator, is an infinitesimal generator of a $C_{0}-$semigroup $\{T(t)\}_{t\geq 0}$ on $H^{s}_{0}(\mathbb{T}).$ It is important to observe that
	$A^{\ast}_{\mu}=-A_{\mu}$  	is the infinitesimal generator of a
	$C_{0}$-semigroup $\{U_{\mu}(t)^{\ast}\}_{t\geq 0},$  with domain of $A^{\ast}_{\mu}$ dense in $L_{0}^{2}(\mathbb{T}),$ and
	$U_{\mu}(t)^{\ast}=U_{\mu}(-t).$
\end{proof}

In order to stabilize the  equation \eqref{atabilizationL2} in $H^{s}_{0}(\mathbb{T})$, we
employ a simple feedback control law, $Ku=-GG^{\ast}u.$ The following theorem
says that the trivial solution, (u=0) of equation  \eqref{atabilizationL2} with this
feedback control law is exponentially asymptotically stable when $t$ goes to infinity.

\begin{thm}\label{st35}
	Let $\alpha>0,$ $\mu\in\mathbb{R},$ $g$ as in \eqref{gcondition},  and  $s\geq 0$ be given. There exist positive constans $M=M(\alpha, \mu, g, s)$ and $\gamma=\gamma(g),$ such that for any
	$u_{0}\in H_{0}^{s}(\mathbb{T}),$  the unique solution $u$
	of \eqref{atabilizationL2} with $K=-GG^{\ast}$ satisfies
	\begin{equation}\label{c5}
	\|u(\cdot,t)\|_{H_{0}^{s}(\mathbb{T})}\leq M
	e^{-\gamma t}\|u_{0}\|_{H_{0}^{s}(\mathbb{T})},\;\;\;\text{for all}\;\;
	t\geq 0.
	\end{equation}
\end{thm}

\begin{proof} 
	We prove this theorem in five steps. 
	
	\noindent
{\bf Step 1.} 
	First we prove the case $s=0.$ In this case we use a procedure similar to \cite{1, Russell and Zhang}. Let $T>0$ be given and assume  $u_{0}\in H_{0}^{3}(\mathbb{T}).$
 Theorem \ref{solsta1} implies that the solution $u$ of the IVP
	\begin{equation}\label{stabilizationL21}
	\left \{
	\begin{array}{l l}
	\partial_{t}u-\alpha \mathcal{H}\partial^{2}_{x}u-\partial^{3}_{x}u+2\mu\partial_{x} u=-GG^{\ast}u,&  t>0,\;\;x\in \mathbb{T}\\
	u(x,0)=u_{0}(x), & x\in \mathbb{T},
	\end{array}
	\right.
	\end{equation}
	satisfies $u\in C([0,\infty);H_{0}^{3}(\mathbb{T}))
	\cap C^{1}([0,\infty);L^{2}_{0}(\mathbb{T})).$ It means $u(\cdot,t)\in H_{0}^{3}(\mathbb{T}),$ for all $t\geq 0$ and in particular, for $t=T.$ Now we consider the IVP
	\begin{equation}\label{stabilizationL22}
	\left \{
	\begin{array}{l l}
	\partial_{t}w-\alpha \mathcal{H}\partial^{2}_{x}w-\partial^{3}_{x}w+2\mu\partial_{x} w=Gh,&  t\in (0,T),\;\;x\in \mathbb{T}\\
	w(x,0)=0, & x\in \mathbb{T}.
	\end{array}
	\right.
	\end{equation}
	
  Remark \ref{sol4} implies that there exists a unique $h\in
	L^{2}([0,T]; H_{0}^{3}(\mathbb{T}))$ such that the unique solution
	$w\in C([0,\infty);H_{0}^{3}(\mathbb{T}))
	\cap C^{1}([0,\infty);L^{2}_{0}(\mathbb{T}))$
	of equation  \eqref{stabilizationL22} satisfies
	$ w(x,T)=u(x,T)$ for all $x\in \mathbb{T},$ and there exists a positive constant $\nu=\nu(g)$ such that
	\begin{equation}\label{sta1}
	\|h\|_{L^{2}([0,T];H_{0}^{3}(\mathbb{T}))} \leq
	\nu\; \|u(x,T)\|_{H_{0}^{3}(\mathbb{T})}.
	\end{equation}

	On the other hand, note that  $u_{0}\in H_{0}^{3}(\mathbb{T})\subset L_{0}^{2}(\mathbb{T}),$
	therefore Theorem \ref{solsta1} implies that $ u \in
	C([0,\infty);L_{0}^{2}(\mathbb{T}))$ is a solution of equation \eqref{stabilizationL21}.
	Furthermore, Remark \ref{sol4} implies that 
	$ h\in
	L^{2}([0,T]; L_{0}^{2}(\mathbb{T}))$ and the solution
	$ w\in C([0,\infty);L_{0}^{2}(\mathbb{T}))$
	of equation  \eqref{stabilizationL22} satisfies
	$ w(x,T)= u(x,T),\; \text{for all}\; x\in \mathbb{T},$ with
	\begin{equation}\label{sta12}
	\begin{split}
	\| h\|_{L^{2}([0,T];L_{0}^{2}(\mathbb{T}))}& \leq
	\nu\; \| u(x,T)\|_{L_{0}^{2}(\mathbb{T})}.
	\end{split}
	\end{equation}

	Now, multiplying the first equation in \eqref{stabilizationL21} by $\bar u$ and integrating with respect to $x,$ it follows that
	\begin{equation}\label{insta1}
	\int_{\mathbb{T}}\partial_{t}u \; \bar u\;dx-\int_{\mathbb{T}} \alpha
	\mathcal{H}\partial_{x}^{2}u\; \bar u dx- \int_{\mathbb{T}} \partial_{x}^{3}
	u\; \bar u\;dx + \int_{\mathbb{T}}2\mu\;\partial_{x}u\;\bar u\;dx=
	\int_{\mathbb{T}}-GG^{\ast} u\;\bar u dx.
	\end{equation}
	
	Integrating by parts,  using the Parseval's identity  and the fact that the operator $G$ is self-adjoint on $L^{2}_{0}(\mathbb{T}),$ it is easy to obtain from \eqref{insta1} that
	\begin{equation}\label{insta7}
	\dfrac{1}{2}\dfrac{d}{dt}\left(\|u(\cdot,t)\|^{2}_{L^{2}_{0}(\mathbb{T})}\right)=
	-\|Gu(\cdot,t)\|^{2}_{L^{2}_{0}(\mathbb{T})},\;\;\;\text{for all}\;t>0.
	\end{equation}

	Now integrating \eqref{insta7} with respect to the variable $t$ from $0$ and $T,$ we get
	\begin{equation}\label{insta8}
	\dfrac{1}{2}\|u(T)\|^{2}_{L^{2}_{0}(\mathbb{T})}
	-\dfrac{1}{2}\|u_{0}\|^{2}_{L^{2}_{0}(\mathbb{T})}=
	-\|Gu\|^{2}_{L^{2}((0,T);L^{2}_{0}(\mathbb{T}))}.
	\end{equation}
	
	On the other hand, multiplying \eqref{stabilizationL22} by $\bar u$
	and integrating with respect to the $x-$variable,  we get
	\begin{equation}\label{instalne1}
	\int_{\mathbb{T}}\partial_{t}w \; \bar u\;dx-\int_{\mathbb{T}}\left( \alpha
	\mathcal{H}\partial_{x}^{2}w\; \bar u +  \partial_{x}^{3}
	w\; \bar u\; - 2\mu\;\partial_{x}w\;\bar u\right)\;dx=
	\int_{\mathbb{T}}Gh\;\bar u dx,\;\;\text{for all}\;t>0.
	\end{equation}
	
	Using integration by parts in the second term of \eqref{instalne1}
	we get
	\begin{equation}\label{instalne2}
	\int_{\mathbb{T}}\partial_{t}w \; \bar u\;dx-\int_{\mathbb{T}}w\; \overline{\left( -\alpha
		\mathcal{H}\partial_{x}^{2}u -  \partial_{x}^{3}
		u\; + 2\mu\;\partial_{x}u\right)}\;dx=
	\int_{\mathbb{T}}Gh\;\bar u dx,\;\;\text{for all}\;t>0.
	\end{equation}
	
	Integrating \eqref{instalne2} with respect to $t$ from $0$ and $T,$
	and using integration by parts,
	we obtain
$$\int_{\mathbb{T}}w(x,T)\;  \bar u(x,T)\;dx-\int_{0}^{T}\int_{\mathbb{T}}w \; \overline{\left( \partial_{t}u-\alpha
\mathcal{H}\partial_{x}^{2}u -  \partial_{x}^{3}
u + 2\mu\partial_{x}u\right)}\;dx\;dt=
\int_{0}^{T}\int_{\mathbb{T}}Gh \;\bar u \;dx\;dt.$$

	Observe that $u$ is a solution of equation \eqref{stabilizationL21}. Thus
$$\int_{\mathbb{T}}u^{2}(T) \;dx+\int_{0}^{T}\int_{\mathbb{T}}w\; \overline{(GG^{\ast}u)}\;dx\;dt=
\int_{0}^{T}\int_{\mathbb{T}}Gh\;\bar u \;dx\;dt.$$
	
	Using that the solution $u$ is real, the	operator $G$ is self-adjoint on $L^{2}_{0}(\mathbb{T}),$ and the Cauchy-Shwartz inequality, we get
	\begin{equation}\label{instalne5}
	\begin{split}
	\|u(\cdot,T)\|^{2}_{L^{2}_{0}(\mathbb{T})}
	&\leq
	\|h-Gw\|_{L^{2}((0,T) ;L_{0}^{2}( \mathbb{T}))} \; \|Gu\|_{L^{2}((0,T)
		;L_{0}^{2}( \mathbb{T}))}.\\
	\end{split}
	\end{equation}
	
	From \eqref{sta12}, we have
	\begin{equation}\label{instalne6}
	\begin{split}
	\|h-Gw\|_{L^{2}((0,T) ;L_{0}^{2}( \mathbb{T}))}
	&\leq
	\nu \|u(T)\|_{L_{0}^{2}( \mathbb{T})}+c \;
	\left(\int_{0}^{T}
	\|w(\cdot,t)\|^{2}_{L_{0}^{2}( \mathbb{T})}\;dt\right)^{\frac{1}{2}}.
	\end{split}
	\end{equation}
	
		Also, observe that
	\begin{equation}\label{instalne7}
	\begin{split}
	\|w(\cdot,t)\|^{2}_{L_{0}^{2}( \mathbb{T})} &\leq
	\left\|\int_{0}^{t}U_{\mu}(t-t')Gh(\cdot,t')\;dt'\right\|^{2}_{L_{0}^{2}( \mathbb{T})}\\
	&\leq c^{2}\; T \left[\left( \int_{0}^{T}
	\left\|h(\cdot,t')\right\|^{2}_{L_{0}^{2}( \mathbb{T})}\;dt'
	\right)^{\frac{1}{2}}\right]^{2}\\
	&\leq c^{2}\; T \|h\|_{L^{2}((0,T);L_{0}^{2}(\mathbb{T}))}^{2}\\
	&\leq c^{2}\; T\; \nu^{2} \|u(T)\|^{2}_{L^{2}_{0}(\mathbb{T})}.
	\end{split}
	\end{equation}		
	
	It follows from \eqref{instalne6} and \eqref{instalne7} that
	\begin{equation}\label{instalne8}
	\begin{split}
	\|h-Gw\|_{L^{2}((0,T) ;L_{0}^{2}( \mathbb{T}))}
	&\leq c_{g,T}\; \|u(T)\|_{L_{0}^{2}( \mathbb{T})},
	\end{split}
	\end{equation}
	where $\displaystyle{c_{g,T}=max\{\nu,\;c^{2}\; T\; \nu\}}.$
	
	Thus, from \eqref{instalne5} and \eqref{instalne8}, we have
	\begin{equation}\label{instalne9}
	\|u(\cdot,T)\|^{2}_{L^{2}_{0}(\mathbb{T})}
	\leq c_{g,T}\; \|u(T)\|_{L_{0}^{2}( \mathbb{T})} \cdot \|Gu\|_{L^{2}((0,T)
		;L_{0}^{2}( \mathbb{T}))},
	\end{equation}
	which  implies that
	\begin{equation}\label{instalne10}
	-\|Gu\|^{2}_{L^{2}((0,T);L_{0}^{2}( \mathbb{T}))}
	\leq -\frac{1}{c^{2}_{g,T}} \;
	\|u(\cdot,T)\|^{2}_{L^{2}_{0}(\mathbb{T})}.
	\end{equation}
	
	From identity \eqref{insta8} and the inequality \eqref{instalne10}, we obtain
	\begin{equation}
	\left(1+\frac{2}{c^{2}_{g,T}}\right)\|u(T)\|^{2}_{L^{2}_{0}(\mathbb{T})}\leq
	\|u_{0}\|^{2}_{L^{2}_{0}(\mathbb{T})}.
	\end{equation}
	
	Thus, there exists $\rho_{g,T}=\rho \in (0,1)$ such that
$$\|u(T)\|^{2}_{L^{2}_{0}(\mathbb{T})}\leq \rho \;
\|u_{0}\|^{2}_{L^{2}_{0}(\mathbb{T})}, \;\;\;\text{for any} \;\;T>0.$$
	
	Moreover, we can repeat this estimate on successive intervals $[(n-1)T,nT],$ to get 
	\begin{equation}\label{instalne13}
	\begin{split}
	\|u(x,nT)\|^{2}_{L^{2}_{0}(\mathbb{T})}
	&\leq \rho^{n} \;
	\|u_{0}\|^{2}_{L^{2}_{0}(\mathbb{T})}, \;\;\;\text{for any } \;\;T>0,\;n\geq 1,
	\end{split}
	\end{equation}
	where $u$ is the solution of \eqref{stabilizationL21}, and  $\rho=\rho_{g,T}\in (0,1).$
	
	In particular, fixing \; $T>0$ we obtain that  for any $t\geq 0,$ there exists $n\in \mathbb{N}$ such that $nT\leq t\leq (n+1)T.$ From \eqref{insta7} we know that
	the function $t \rightarrow \|u(\cdot,t)\|^{2}_{L^{2}_{0}(\mathbb{T})}, $ with $t\geq 0$ is decreasing. From \eqref{instalne13} there exists $\rho=\rho_{g}\in (0,1)$ such that
	\begin{align*}
	\|u(x,t)\|^{2}_{L^{2}_{0}(\mathbb{T})}&\leq
	\|u(x,nT)\|^{2}_{L^{2}_{0}(\mathbb{T})}\\
	&\leq \rho^{n} \;
	\|u_{0}\|^{2}_{L^{2}_{0}(\mathbb{T})}, \;\;\;\text{for all} \;\;n\geq 1.
	\end{align*}
	
	It is easy to show that if 
$$0<\gamma\leq-\frac{ln(\rho)}{2T},\;\;\;\text{and}\;\;\;M\geq e^{\gamma \;T},$$
	 one has
$$\rho^{n}\leq M^{2}\;e^{-2\gamma \;t},
	\;\;\;\text{for all}\;\;n\in \mathbb{N}.$$

	Therefore,
	\begin{equation}\label{instalne15}
	\|u(x,t)\|_{L^{2}_{0}(\mathbb{T})}
	\leq M\; e^{-\gamma \;t}
	\|u_{0}\|_{L^{2}_{0}(\mathbb{T})}, \;\;\;\text{for all} \;t\geq 0,
	\end{equation}
	and we get the result for smooth initial data in $H_{0}^{3}(\mathbb{T})$.
	We complete the proof for $s=0$ using density arguments.\\

   \noindent
{\bf Step 2.}  Here we consider $s=3.$ In this case we use a similar argument as in Proposition 2.3 of  \cite{14}.
	Let $u$ be the solution of equation \eqref{stabilizationL21} with initial data $u_{0}\in H_{0}^{3}(\mathbb{T})$, then
$$u\in C([0,\infty);H_{0}^{3}(\mathbb{T}))
	\cap C^{1}([0,\infty);L^{2}_{0}(\mathbb{T})).$$
	
	Since
	$H_{0}^{3}(\mathbb{T})\subset L_{0}^{2}(\mathbb{T}),$ then from the  $s=0$ case we have that there exist positive constants
	$M_{1}$ and $ \gamma=\gamma(g)$ independent of $u_{0},$ such that
	\begin{equation}\label{expstabilizationinL21}
	\|u(\cdot,t)\|_{L_{0}^{2}(\mathbb{T})}\leq M_{1}
	e^{-\gamma t}\|u_{0}\|_{L_{0}^{2}(\mathbb{T})},\;\;\;\text{for all}\;\;
	t\geq 0.
	\end{equation}

	On the other hand, differentiating the equation \eqref{stabilizationL21} with respect
	to $t,$ we obtain
$$\partial_{t}(\partial_{t}u)-\alpha \mathcal{H}\partial^{2}_{x}(\partial_{t}u)
-\partial^{3}_{x}(\partial_{t}u)+2\mu\partial_{x}(\partial_{t} u)
=-GG^{\ast}(\partial_{t}u).$$
	
	Therefore, $w:=\partial_{t}u \in C([0,+\infty); L_{0}^{2}(\mathbb{T}))$ is the unique solution of	
	\begin{equation}\label{estabilizacion15}
	\partial_{t}w-\alpha \mathcal{H}\partial^{2}_{x}w
	-\partial^{3}_{x}w+2\mu\partial_{x} w
	=-GG^{\ast}w,\;\;\;t>0,\;\;x\in \mathbb{T},
	\end{equation}	
	with initial data
	\begin{equation}\label{estabilizacion16}
	w(x,0)=w_{0}=\partial_{t}u(x,0)=\alpha \mathcal{H}\partial^{2}_{x}u_{0}
	+\partial^{3}_{x}u_{0}-2\mu\partial_{x} u_{0}
	-GG^{\ast}u_{0} \in L_{0}^{2}(\mathbb{T}),\;x\in \mathbb{T}.
	\end{equation}
	
	Again, from the case $s=0$ applied to equation \eqref{estabilizacion15},   there exist positive constants
	$M_{1}=M_{1}(g)$ and 
	$\gamma=\gamma(g),$ independent of $w_{0},$ such that
	\begin{equation}\label{expstabilizationinL2w}
	\|\partial_{t}u(\cdot,t)\|_{L_{0}^{2}(\mathbb{T})}=
	\|w(\cdot,t)\|_{L_{0}^{2}(\mathbb{T})}\leq M_{1}
	e^{-\gamma t}\|w_{0}\|_{L_{0}^{2}(\mathbb{T})},\;\;\;\text{for all}\;\;
	t\geq 0.
	\end{equation}
	
	Note that, for each $t\geq 0$ 
	\begin{equation}\label{estabilizacion8}
	\begin{split}
	\|u(\cdot,t)\|_{H_{0}^{3}(\mathbb{T})}
	&\leq c_{0}\;\left(\|u(\cdot,t)\|_{L^{2}_{0}(\mathbb{T})}+
	\|\partial_{x}^{3}u(\cdot,t)\|_{L^{2}_{0}(\mathbb{T})}\right).
	\end{split}
	\end{equation}
	
	To estimate the term $\|\partial_{x}^{3}u(\cdot,t)\|_{L^{2}_{0}(\mathbb{T})}$ observe
	that from equation \eqref{stabilizationL21}
$$\partial_{x}^{3}u(\cdot,t)= w -\alpha \mathcal{H}\partial^{2}_{x}u
+2\mu\partial_{x} u +GG^{\ast}u.$$
	
	Thus, for each $t\geq 0$
	\begin{equation}\label{estabilizacion0}
		{\footnotesize
		\begin{split}
		\|\partial_{x}^{3}u(\cdot,t)\|_{L^{2}_{0}(\mathbb{T})} &\leq
		\|w(\cdot,t)\|_{L^{2}_{0}(\mathbb{T})}+\alpha
		\|\mathcal{H}\partial^{2}_{x}u(\cdot,t) \|_{L^{2}_{0}(\mathbb{T})}+
		2|\mu|\|\partial_{x}u(\cdot,t)\|_{L^{2}_{0}(\mathbb{T})}+
		\|GG^{\ast}u(\cdot,t)\|_{L^{2}_{0}(\mathbb{T})}.
		\end{split}}
	\end{equation}
	
	Using Gagliardo-Niremberg inequality (see the Theorem 3.70  in \cite{Aubin}) and Cauchy-Schwartz inequality with $\epsilon,$ we have
	\begin{equation}\label{estabilizacion1}
	{\small
		\begin{split}
		2|\mu|\|\partial_{x}u(\cdot,t)\|_{L^{2}_{0}(\mathbb{T})}&\leq
		2|\mu|\sqrt{2\pi}\|\partial_{x}u(\cdot,t)\|_{L^{\infty}(\mathbb{T})} \\
		&\leq	2|\mu|\sqrt{2\pi}\;c_{1}
		\|\partial_{x}^{3}u(\cdot,t)\|^{\frac{1}{2}}_{L^{2}_{0}(\mathbb{T})}
		\|u(\cdot,t)\|^{\frac{1}{2}}_{L^{2}_{0}(\mathbb{T})}
		\\
		&=	c_{\mu}
		\epsilon \|u(\cdot,t)\|_{L^{2}_{0}(\mathbb{T})}
		+\frac{c_{\mu}}{4\epsilon}\|\partial_{x}^{3}u(\cdot,t)\|
		_{L^{2}_{0}(\mathbb{T})},			
		\end{split}}
	\end{equation}
	where $c_{\mu}=2|\mu|\sqrt{2\pi}\;c_{1}.$
	Also, using that $\mathcal{H}$ is an isometry in $L^{2}_{0}(\mathbb{T}),$ integration by parts and Cauchy-Schwartz inequality with $\epsilon$, we obtain
	\begin{equation}
		\begin{split}
		\|\mathcal{H}\partial^{2}_{x}u(\cdot,t)\|^{2}_{L^{2}_{0}(\mathbb{T})}&=
		\int_{\mathbb{T}}\partial_{x}^{2}u(x,t)\;
		\overline{\partial_{x}^{2}u(x,t)}\;dx\\		
		&\leq \|\partial_{x}u(\cdot,t)\|_{L^{2}_{0}(\mathbb{T})}	
		\|\partial_{x}^{3}u(\cdot,t)\|_{L^{2}_{0}(\mathbb{T})}	\\	
		&\leq \epsilon \|\partial_{x}u(\cdot,t)\|^{2}_{L^{2}_{0}(\mathbb{T})}+	
		\frac{1}{4\epsilon}\|\partial_{x}^{3}u(\cdot,t)\|^{2}_{L^{2}_{0}(\mathbb{T})}.	
		\end{split}
	\end{equation}
	
	Therefore,
	\begin{equation}\label{c4}
		\begin{split}
		\|\mathcal{H}\partial^{2}_{x}u(\cdot,t)\|_{L^{2}_{0}(\mathbb{T})}&\leq	
		c_{2}\;\Big(\epsilon^{\frac{1}{2}} \|\partial_{x}u(\cdot,t)\|_{L^{2}_{0}(\mathbb{T})}+	
		\frac{1}{2\epsilon^{\frac{1}{2}}}\|\partial_{x}^{3}u(\cdot,t)\|
		_{L^{2}_{0}(\mathbb{T})}\Big).	
		\end{split}
	\end{equation}
	
	Using inequality \eqref{estabilizacion1} we obtain from \eqref{c4}
		\begin{equation}\label{estabilizacion2}
		\begin{split}
		\|\mathcal{H}\partial^{2}_{x}u(\cdot,t)\|_{L^{2}_{0}(\mathbb{T})}&\leq	
		c_{3}\;\epsilon^{\frac{3}{2}}
		\|u(\cdot,t)\|_{L^{2}_{0}(\mathbb{T})}
		+\frac{c_{4}}{\epsilon^{\frac{1}{2}}}
		\|\partial_{x}^{3}u(\cdot,t)\|
		_{L^{2}_{0}(\mathbb{T})},			
		\end{split}
		\end{equation} 
	where $c_{3}=c_{2}\; \sqrt{2\pi}\;c_{1},$ and
	$c_{4}=\frac{c_{2}c_{1} \sqrt{2\pi}}{4}+\frac{c_{2}}{2}.$
	Thus, from inequalities \eqref{expstabilizationinL2w}, \eqref{estabilizacion0}, \eqref{estabilizacion1} and \eqref{estabilizacion2}, we obtain 
	\begin{equation}\label{estabilizacion4}
	{\footnotesize
	\begin{split}
	\left(1-\frac{c_{4}\;\alpha}{\epsilon^{\frac{1}{2}}}-
	\frac{c_{\mu}}{4\epsilon}\right) \|\partial_{x}^{3}u(\cdot,t)\|_{L^{2}_{0}(\mathbb{T})}
	&\leq M_{1} e^{-\gamma t}\|w_{0}\|_{L_{0}^{2}(\mathbb{T})}
	+\left(\alpha \;c_{3}\;\epsilon^{\frac{3}{2}}+c_{\mu}\epsilon+
	c_{g}^{2}\right)\|u(\cdot,t)\|_{L^{2}_{0}(\mathbb{T})}\\
	&\leq M_{1} e^{-\gamma t}\|w_{0}\|_{L_{0}^{2}(\mathbb{T})}
	+\left(\alpha \;c_{3}\;\epsilon^{\frac{3}{2}}+c_{\mu}\epsilon+
	c_{g}^{2}\right)M_{1} e^{-\gamma t}\|u_{0}\|_{L_{0}^{2}(\mathbb{T})}.
	\end{split}}
	\end{equation}
	
	Therefore, taking $\epsilon>0$ large enough such that $1-\frac{c_{4}\;\alpha}{\epsilon^{\frac{1}{2}}}-
	\frac{c_{\mu}}{4\epsilon}>0$ we infer that there exists a positive
	constant $c=c_{\alpha, \mu, g},$ independent of $u_{0},$ and $w_{0}$ such that	
	\begin{equation}\label{estabilizacion5}
	\begin{split}
	\|\partial_{x}^{3}u(\cdot,t)\|_{L^{2}_{0}(\mathbb{T})}
	&\leq\;c\; M_{1} e^{-\gamma t} \left(\|w_{0}\|_{L_{0}^{2}(\mathbb{T})}
	+\|u_{0}\|_{L_{0}^{2}(\mathbb{T})}\right).
	\end{split}
	\end{equation}
	
	Also, note that
	\begin{equation}\label{estabilizacion6}
	\begin{split}
	\|w_{0}\|_{L_{0}^{2}(\mathbb{T})}
	&\leq \alpha\; \|\partial_{x}^{2}u_{0}\|_{L_{0}^{2}(\mathbb{T})}
	+\|\partial_{x}^{3}u_{0}\|_{L_{0}^{2}(\mathbb{T})}
	+2|\mu|\;\|\partial_{x}u_{0}\|_{L_{0}^{2}(\mathbb{T})}
	+c_{g}^{2}\|u_{0}\|_{L_{0}^{2}(\mathbb{T})}\\
	&\leq c_{6}\;\|u_{0}\|_{L_{0}^{2}(\mathbb{T})},
	\end{split}
	\end{equation}
	where $c_{6}=c_{5}(\alpha+1+2|\mu|+c_{g}^{2}).$
	Thus from \eqref{estabilizacion5} and \eqref{estabilizacion6}, we have
	\begin{equation}\label{estabilizacion7}
	\begin{split}
	\|\partial_{x}^{3}u(\cdot,t)\|_{L^{2}_{0}(\mathbb{T})}
	&\leq M_{2} e^{-\gamma t}\|u_{0}\|_{L_{0}^{2}(\mathbb{T})},
	\end{split}
	\end{equation}
	where $M_{2}=c\; M_{1}(c_{6}+1).$
	
	Now, from \eqref{expstabilizationinL21},  \eqref{estabilizacion8} and \eqref{estabilizacion7}, we get
	\begin{equation}\label{estabilizacion9}
	\begin{split}
	\|u(\cdot,t)\|_{H_{0}^{3}(\mathbb{T})}
	&\leq c_{0}\;\left(M_{1}
	e^{-\gamma t}\|u_{0}\|_{L_{0}^{2}(\mathbb{T})}+
	M_{2} e^{-\gamma t}\|u_{0}\|_{L_{0}^{2}(\mathbb{T})}\right)\\
	&\leq c_{0}\;\left(M_{1} +
	M_{2}\right) e^{-\gamma t}\;c_{5}\;\|u_{0}\|_{H_{0}^{3}(\mathbb{T})}\\
	& \leq M\;e^{-\gamma t}\|u_{0}\|_{H_{0}^{3}(\mathbb{T})},\;\;\;
	\text{for all}\;\;t\geq 0,
	\end{split}
	\end{equation}
	where  $M=M(\alpha, \mu, g)=c_{0}\;(M_{1}+M_{2})\;c_{5},$
	and $\gamma=\gamma(g)$ are positive constants independent of $u_{0}.$
	
 \noindent
{\bf Step 3.}  Using induction and similar arguments as above, we prove that inequality 
	\eqref{c5} holds for $s=3n,$ with $n\in \mathbb{N}.$
	
\noindent
{\bf Step 4.} 	We consider $0<s<3.$ In this case we use  the Real Interpolation Method, especifically the K-method of Interpolation, (see Definition 2.4.3, and  Theorem 3.1.2 in  Bergh and  Lofstrom \cite{Bergh and Lofstrom}).
	From  Corollary 1.111  in Triebel
	\cite{Triebel}  we know that the space of interpolation between $L_{0}^{2}(\mathbb{T})$ and
	$H^{3}_{0}(\mathbb{T})$ is
$$(L_{0}^{2}(\mathbb{T}),H_{0}^{3}(\mathbb{T}))_{\theta,2}=
H_{0}^{3\theta}(\mathbb{T}),$$
	where $0<\theta<1.$
	Therefore, interpolating \eqref{instalne15} and 
	\eqref{estabilizacion9} we get that there exists $M=M(\alpha, \mu, g, \theta),$ and $\gamma=\gamma(g)$ such that
$$\| u(\cdot,t) \|_{H_{0}^{3\theta}(\mathbb{T})}
\leq   M e^{-\gamma\;t}
\| u_{0} \|_{H_{0}^{3\theta}(\mathbb{T})}, 
\;\;\text{for all}\;t\geq 0,$$
	where $0<\theta<1,$ and $u$ is the solution of \eqref{atabilizationL2} with $K=-GG^{\ast}.$ Thus, denoting $s=3\theta,$ we obtain the result.

\noindent
{\bf Step 5.} 	Finally, using an induction argument and computations  similar to those in the previous cases we can prove the following claim.
	
\noindent	
{\bf Claim:} For $0< \rho < 1,$ and $n\in \mathbb{N}\cup \{0\},$ there exist positive constants 
	$M=M(\alpha, \mu, g, n,\rho)$ and $\gamma=\gamma(g),$ such that for any
	$u_{0}\in H_{0}^{3n+3\rho}(\mathbb{T}),$  the unique solution $u$
	of \eqref{atabilizationL2} with $K=-GG^{\ast}$ satisfies
	$$\|u(\cdot,t)\|_{H_{0}^{3n +3\rho}(\mathbb{T})}\leq M
	e^{-\gamma t}\|u_{0}\|_{H_{0}^{3n +3\rho}(\mathbb{T})},\;\;\;\text{for all}\;\;
	t\geq 0.$$
	
	Note that, for $s\geq0$ given, there exist $n\in \mathbb{N}\cup \{0\}$ and $0\leq \rho \leq 1,$ such that $s=3n+3\rho.$
	Therefore,
 inequality \eqref{c5} for the other values of $s$ follows from the claim and the result obtained in the third step. This complete the proof the Theorem \ref{st35}.
\end{proof}

Observe that Theorem \ref{st351} is a direct consequence of Theorem \ref{st35}.

\subsection{Stabilization of the linear Benjamin equation with an arbitrary decay rate}

In this subsection, we show that it is possible to choose an appropriate linear feedback control law such that the decay rate of the resulting closed-loop system \eqref{atabilizationL2} is as large as one desires. 

Let $T>0$ be any fixed number.
For $\lambda >0$ and $s\geq 0$ given, we  define the operator
\begin{equation}\label{st1}
L_{\lambda}\phi=\int_{0}^{T}e^{-2\lambda \tau}\;U_{\mu}(-\tau)GG^{\ast}
U_{\mu}(-\tau)^{\ast}\phi\;d\tau,
\;\;\;\text{for all}\;\;\phi\in H_{p}^{s}(\mathbb{T}).
\end{equation}

With an analogous argument as in Lemma 2.4 of \cite{14} 
we can prove the following properties of this operator.

\begin{lem}\label{st9}
The operator	$L_{\lambda}:H_{p}^{s}(\mathbb{T})\longrightarrow
H_{p}^{s}(\mathbb{T})$ is linear and bounded.   Moreover,
	$L_{\lambda}$ is an isomorphism from $H_{0}^{s}(\mathbb{T})$
	onto $H_{0}^{s}(\mathbb{T})$, for all $s\geq 0.$
\end{lem}

\begin{rem}\label{st36}
	Lemma  \ref{st9} implies that there exists a positive constant $C=C(\delta,s,\lambda,T,g)$ such that
$$\|L_{\lambda}^{-1}\psi\|_{H_{0}^{s}(\mathbb{T})}
\leq\;C\;
\|\psi\|_{H_{0}^{s}(\mathbb{T})},\;\;
\text{for all}\;\;\psi\in H_{0}^{s}(\mathbb{T}).$$
\end{rem}

Choosing the feedback control law in system \eqref{atabilizationL2}  as
\begin{equation}\label{feedback}
Ku:= \left\{
\begin{array}{lcl}
-K_{\lambda}u=-G G^{\ast} L_{\lambda}^{-1}u, & \mbox{if} &  \lambda>0
\\
&           &          \\
-K_{0}u=-G G^{\ast}u,
& \mbox{if} &\lambda=0,
\end{array}
\right.
\end{equation}
we can rewrite  the resulting closed-loop system in the following form
\begin{equation}\label{st21}
\left \{
\begin{array}{l l}
\partial_{t}u-\alpha \mathcal{H}\partial^{2}_{x}u-\partial^{3}_{x}u+2\mu\partial_{x} u=-K_{\lambda}u,&  t>0,\;\;x\in \mathbb{T}\\
u(x,0)=u_{0}(x), & x\in \mathbb{T},
\end{array}
\right.
\end{equation}
where  $\lambda\geq 0$   and $K_{\lambda}$ is a bounded linear operator on
$H^{s}_{0}(\mathbb{T})$ with $s\geq 0.$ 
We have the following result.

\begin{thm}\label{st37}
	Let $\alpha>0,$ $\mu \in \mathbb{R},$ $s\geq 0$ and $\lambda>0$ be given.
	For any $u_{0}\in H_{0}^{s}(\mathbb{T}),$ the system \eqref{st21} admits a unique solution $u\in C([0,+\infty), H_{0}^{s}(\mathbb{T})).$ Moreover,
	there exists $M=M(g,\lambda,\delta, \alpha, \mu, s)>0$ such that
	$$\|u(\cdot,t)\|_{H_{0}^{s}(\mathbb{T})}\leq
	M\;e^{-\lambda\;t}\|u_{0}\|_{H_{0}^{s}(\mathbb{T})},
	\;\;\;\text{for all}\;\;t\geq0.$$
\end{thm}

\begin{proof}
	As $K_{\lambda}$ is a bounded linear operator	the same argument used in Theorem \ref{solsta1} shows that
	for  $u_{0}\in H_{0}^{s}(\mathbb{T})$ the problem \eqref{st21}
	has a unique solution
	$u\in C([0,\infty);H_{0}^{s}(\mathbb{T}))$ for all $s \geq 0.$
	We denote by $\{T_{\lambda}(t)\}_{t\geq 0}$ the $C_{0}-$semigroup  on $H^{s}_{0}(\mathbb{T})$ with infinitesimal generator 	
	$A_{\mu}-K_{\lambda}.$

	The $s=0$ case follows from Theorem 2.1 in \cite{Slemrod}. The others cases of $s$	are proved as in Theorem \ref{st35}.
\end{proof}

Finally, observe that Theorem \ref{estabilization} is a direct consequence of Theorem \ref{st37}.

\section{Concluding Remarks}\label{conc-rem}

We proved that the linearized Benjamin equation with periodic boundary conditions is exactly controllable and  exponentially stabilizable with any given decay rate in $H_{p}^{s}(\mathbb{T})$ with $s\geq0$. These results are in accordance with the controllability and stabilization results for the linearized BO and  the KdV equations respectively obtained in \cite{1} and \cite{10}.  The Benjamin equation has a combination of the KdV term $\partial_{x}^{3}u$ and the BO term $\alpha \mathcal{H} \partial_{x}^{2}$ in its linear part. Recently, using propagation of compactness, unique continuation property  and  propagation of smoothness, Laurent, Rosier and Zhang \cite{14} proved that the nonlinear  KdV equation is globally exactly controllable and globally exponentially stabilizable. Very recently, similar results  for the nonlinear BO equation are proved by  Laurent, Linares and Rosier \cite{Laurent Linares and Rosier}. Therefore,
it is natural to ask if these controllability and stabilizability results are valid for the   nonlinear Benjamin equation as well. Taking idea from \cite{Laurent}, \cite{14} and \cite{Laurent Linares and Rosier}, we plan to  derive  propagation of compactness, unique continuation property and  propagation of smoothness  results for the solutions of the Benjamin equation in some adequate Bourgain's spaces in order to provide an affirmative answer to the question posed above. This work is in progress.


\subsection*{Acknowledgements}
F. Vielma is supported by FAPESP, Brazil (grant no. 2015/06131-5). The authors would like to thank Prof. Felipe linares, Prof. Ademir Pastor, and Prof. Lionel Rosier  for many helpful discussions and suggestions.


\begin{thebibliography}{99}
\small

\bibitem{15} Albert J. P., Bona J. L., Restrepo J. M., {\em Solitary waves solutions of the Benjamin equation}, SIAM J. Appl. Math. {\bf 59}  6 (1999) 2139--2161.

\bibitem{16} Angulo J., {\em Existence and stability of solitary wave solutions of the Benjamin equation}, J.  Differential Equations {\bf 152}  1 (1999)  136--159.


\bibitem{Aubin} Aubin T., {\em Nonlinear Analysis on Manifolds
Monge-Ampere Equations}, Grundlehren Mathematischen Wissenschaften 252, Springer Verlag,  New York Inc., Heidelberg Berlin  (1982).




\bibitem{Ball and Slemrod} Ball J. M., Slemrod M., {\em Nonharmonic Fourier Series and the Stabilization of
Distributed Semi-linear Control Systems}, Comm. Pure Appl. Math. {\bf 32} 4 (1979) 555--587.

\bibitem{Bergh and Lofstrom} Bergh J., Lofstrom J., {\em Interpolation spaces an introduction}, A series of comprenhensive studies in Mathematics, First Edition,  Springer-Verlag Berlin Heidelberg New York (1976).

\bibitem{5} Benjamin B. {\em A new kind of solitary waves}, J. Fluid Mech. \textbf{245} (1992) 401--411.


\bibitem{3} Cazenave T., Haraux H., {\em An introduction to Semilinear Evolutions Equations}, John Wiley and Sons, Inc., Revised edition, Clarendon Press-Oxford (1998).

\bibitem{18} Chen H., Bona J. L., \textit{Existence and asymptotic properties of solitary-wave solutions of Benjamin-type equations}, Adv, Diff. Eqns. \textbf{3} 1 (1998) 51--84.

\bibitem{20} Chen W., Guo Z., Xiao J., \textit{Sharp well-posedness for the Benjamin equation}, Nonlinear Anal. {\bf 74} 17 (2011) 6209--6230.

\bibitem{Coron} Coron J.-M., \textit{Control and Nonlinearity}, Amer. Math. Soc., Mathematical surveys and Monographs {\bf 136}  (2007).

\bibitem{Coron Crepau}  Coron J. -M., E. Cr\'epeau, {\em Exact boundary controllability of a nonlinear KdV equation
with a critical length}, J. Eur. Math. Soc. {\bf 6} (2004) 367--398.

\bibitem{Zeidler} Eberhard Z., {\em Nonlinear Functional Analysis and its Applications II/A, Linear Monotone Operators}, Springer (1992).

\bibitem{9} Heil C., {\em A Basis Theory Primer}, Expanded Edition, Applied and Numerical Harmonic Analysis, Birkhauser, Editorial Advisory Board (2011).

\bibitem{8} Ingham A. E., {\em Some trigonometrical Inequalities with applications in the theory of series}, Math. Z. {\bf 41} (1936) 367--379.

\bibitem{6} Iorio R. J. Jr. and Magalhes V., {\em Fourier Analysis and Partial Differential Equations}, Cambrige Universiy Press (2001).

\bibitem{7} Komornik V. and Loreti P., {\em Fourier Series in Control Theory}, Springer Monographs in Mathematics (2005).

\bibitem{19} Kozono H., Ogawa T., Tanisaka H., \textit{Well-posedness for the Benjamin equations}, J. Korean Math. Soc. {\bf 38} 6 (2001) 1205--1234.

\bibitem{Laurent} Laurent C., {\em Global controllability and stabilization for the nonlinear Schrodinger equation on an interval}, ESAIM Control Optim. Cal. Var. {\bf 16} 2 (2010) 356--379.

\bibitem{14} Laurent C., Rosier L., Zhang B., {\em Control and Stabilization of the Korteweg-de Vries Equation on a Periodic Domain}, Comm. Partial Differential Equations,  {\bf 35} 4 (2010) 707--744.

\bibitem{Laurent Linares and Rosier} Laurent C., Linares F.,  Rosier L., {\em Control and Stabilization of
the Benjamin-Ono Equation in $L^{2}(\mathbb{T})$}, Arch. Rational Mech. Anal. {\bf 218} 3 (2015) 1531-1575.

\bibitem{1} Linares F., Ortega J. H., {\em On the controllability and stabilization of the linearized Benjamin-Ono equation}, ESAIM: Cont. Opt. Cal. Var. {\bf 11} 2 (2005), 204--218.

\bibitem{Linares Rosier} Linares F., Rosier L., \textit{Control and Stabilization of the Benjamin-Ono Equation on a Periodic Domain}, Trans.  Amer. Math. Soc., {\bf 367} 7 (2015) 4595--4626.

\bibitem{23} Linares F., \textit{$L^{2}$-global well-posedness of the initial value problem associated to the Benjamin equation}, J. Differential equations {\bf 152} 2 (1999) 377--393.

\bibitem{Menzala Vasconcellos Zuazua} Menzala G. P., Vasconcellos C. F.,  Zuazua E., {\em Stabilization of the Korteweg-de Vries
equation with localized damping}, Quart. Appl. Math., {\bf 60} 1 (2002) 111--129.

\bibitem{Micu Ortega Rosier and Zhang} Micu S., Ortega J., Rosier L., Zhang B-Y.,{\em Control and Stabilization of a family
of Boussinesq Systems}, Disc. and Cont. Dyn. Syst. {\bf 24} 2 (2009) 273--313.

\bibitem{2} Pandey J. N., {\em The Hilbert Transform of Schwartz Distributions and Applications}, John Wiley and Sons, Inc., Carleton University (1996).

\bibitem{4} Pazy A., {\em Semigroups of Linear Operators and Applications to Partial Differential Equations}, Springer-Verlag, New York Inc (1983).

\bibitem{Rosier 1}  Rosier L., {\em Exact boundary controllability for the Korteweg-de Vries equation on a bounded
domain}, ESAIM: Cont. Optim. Calc. Var., {\bf 2} (1997) 33--55.

\bibitem{Rosier and Zhang 2} Rosier L., Zhang B. -Y., {\em Global stabilization of the generalized Korteweg-de Vries equation}, SIAM J. Cont. Optim., {\bf 45} 3 (2006) 927--956.

\bibitem{Rudin} Rudin W., {\em Functional Analysis}, Mc-Graw Hill, Inc. Second Edition (1991).

\bibitem{Russell} Russell D. L., {\em Controllability and Stabilizabilization Theory for Linear Partial Differential Equations: Recent Progress and Open Questions},
SIAM Review, {\bf 20} 4 (1978) 639--739.

\bibitem{Russell and Zhang} Russell D. L., Zhang B., {\em  Controllability and Stabilizability of the Thrid-Order Linear Dispertion Equation on a Periodic Domain}, SIAM J. Cont. and Optm,  {\bf 31} 3 (1993) 659--676.

\bibitem{10} Russell D. L., Zhang B., {\em Extact Controllability and Stabilizability of the Korteweg-De Vries Equation}, Trans. Amer. Math. Soc.,  {\bf 348} 9 (1996) 3643--3672.

\bibitem{Slemrod} Slemrod, M.,   \textit{A note on complete controllability and stabilizability for linear control systems in Hilbert space}, SIAM J. Control,  {\bf 12} 3 (1974) 500--508.

\bibitem{22} Shi S.  and Junfeng L., \textit{Local well-posedness for periodic Benjamin equation with small data}. Boundary value Problems a SpringerOpen Journal (2015) 2015:60.


\bibitem{Triebel} Triebel H., {\em Theory of Functions Spaces III},
Monographs in Mathematics, Vol. 100, Birkhauser Verlag, Basel-Boston-Berlin (2006).

\bibitem{Zhang 1} Zhang B. -Y., {\em Exact boundary controllability of the Korteweg-de Vries equation}, SIAM J. Cont.
Optim., {\bf 37} 2 (1999), 543--565.
\end{thebibliography}

\end{document}